\documentclass{article}

\usepackage{amsmath,amsthm,amssymb,mathtools}  
\usepackage{dsfont}				

\usepackage{enumerate}
\usepackage{graphicx}			
\usepackage{subcaption}			
\usepackage[font=small]{caption}	

\usepackage[latin1]{inputenc}   

\usepackage{verbatim}		

\usepackage[
disable,					
textsize=scriptsize,textwidth=4.2cm]{todonotes}	
\newcommand{\MAYBE}[1]{\todo[color=orange!60]{#1}}  
\newcommand{\SELF}[1]{\todo[color=green!40]{#1}} 
\newcommand{\OMIT}[1]{\todo[color=gray!30]{#1}}  
\newcommand{\CITE}[1]{\todo[color=cyan!30]{#1}}  

\newcommand{\SELFL}[1]{\reversemarginpar\todo[color=green!50]{#1}} 
\newcommand{\OMITL}[1]{\reversemarginpar\todo[color=gray!40]{#1}}  
\newcommand{\CITEL}[1]{\reversemarginpar\todo[color=green!20]{#1}}  

\newcommand{\SELFR}[1]{\normalmarginpar\todo[color=green!40]{#1}} 
\newcommand{\OMITR}[1]{\normalmarginpar\todo[color=gray!40]{#1}}  
\newcommand{\CITER}[1]{\normalmarginpar\todo[color=green!20]{#1}}

\usepackage[hyperindex,breaklinks]{hyperref}	
\hypersetup{colorlinks=true, linkcolor=blue, citecolor=blue}

\usepackage[nameinlink]{cleveref} 	
\usepackage{tikz-cd}			
\usepackage{bm}					
\usepackage{tensor}				
\usepackage{cmll}				
\usepackage{stmaryrd}			

\usepackage{makecell}			
\usepackage{booktabs}			
\renewcommand{\arraystretch}{2}	

\newtheorem{theorem}{Theorem}[section]
\newtheorem{proposition}[theorem]{Proposition}
\newtheorem{corollary}[theorem]{Corollary}

\theoremstyle{definition}
\newtheorem{definition}[theorem]{Definition} 

\newtheorem{example}[theorem]{Example}

\newtheoremstyle{named}{}{}{\itshape}{}{\bfseries}{.}{.5em}{\thmnote{#3}}
\theoremstyle{named}

\def\N			{\mathds{N}}
\def\Z			{\mathds{Z}}
\def\R			{\mathds{R}}
\def\C			{\mathds{C}}
\def\Id			{\mathds{1}}

\def\VV			{\mathcal{V}}

\def\II			{\mathcal{I}}		
\def\MM			{\mathcal{M}}		

\def\ii			{\mathbf{i}}				
\def\jj			{\mathbf{j}}				
\def\kk			{\mathbf{k}}				
\def\ll			{\mathbf{l}}				
\def\mm			{\mathbf{m}}

\def\rr			{\mathbf{r}}				
\def\ss			{\mathbf{s}}				
\def\tt			{\mathbf{t}}

\def\meet		{\texttt{meet}}
\def\join		{\texttt{join}}

\def\wrt					{w.r.t.\ }
\def\ie						{i.e.\ }
\def\eg						{e.g.\ }
\def\resp					{resp.\ }

\def\im	 					{\mathrm{i}}				

\newcommand\inner[1] 		{\langle #1 \rangle}		
\newcommand\pairs[2] 		{|#1>#2|}					
\newcommand\ord[1]			{\overline{#1}} 
\def\lcontr					{\mathbin{\lrcorner}}		
\def\rcontr					{\mathbin{\llcorner}}		
\def\glcontr				{\rfloor}					
\def\grcontr				{\lfloor}					
\def\llcontr {\mathbin{\raisebox{\depth}{\scalebox{1}[-1]{$\lnot$}}}}  
\def\PP						{\mathcal{P}}				
\newcommand{\lH}[2][\star]	{\tensor[^{#1}]{{#2}}{}}	
\newcommand{\rH}[2][\star]	{{#2}^{#1}{}}				
\newcommand{\lHB}[1]		{\lH[\star_B]{#1}}			
\newcommand{\rHB}[1]		{\rH[\star_B]{#1}}			
\newcommand\hhat[2]			{#1\string^ {}^{#2}} 
\newcommand{\grade}[1]		{|#1|}						
\newcommand\comp[2] 		{(#1)_{#2}}		
\newcommand{\adj}[1]		{#1^\dagger}				
\newcommand{\vf}[1]			{\nu(#1)}					
\def\pperp					{\simperp}					
\def\ext					{e}						

\DeclareMathOperator{\Span}{span}

\DeclareMathOperator{\Img}{Im}

\begin{document}

\title{Multivector Contractions Revisited, Part I}

\author{Andr\'e L. G. Mandolesi 
               \thanks{Instituto de Matemática e Estatística, Universidade Federal da Bahia, Av. Adhemar de Barros s/n, 40170-110, Salvador - BA, Brazil. E-mail: \texttt{andre.mandolesi@ufba.br}}}
               
\date{\today \SELF{v3.1 reorganizado}}

\maketitle

\abstract{
We reorganize, simplify and expand the theory of contractions or interior products of multivectors, and related topics like Hodge star duality. 
Many results are generalized and new ones are given, like: 
geometric characterizations of blade contractions and regressive products,
higher-order graded Leibniz rules,
determinant formulas,
improved complex star operators, etc.
Different contractions and conventions found in the literature are discussed and compared, in special those of Clifford Geometric Algebra.
Applications of the theory are developed in a follow-up paper.

\vspace{.5em}

\vspace{3pt}

}

\section{Introduction}

Contractions, interior products or inner derivatives of multivectors or forms date back, in essence, to Grassmann \cite{Grassmann1995}, 
and have since been used in Differential Geometry \cite{Abraham1988,Gallier2020}, Physics \cite{Hestenes2005,Oziewicz1986}, Computer Science \cite{BayroCorrochano2018,Dorst2007}, etc.
Still, they are often seen as somewhat obscure operations.

A difficulty is the various kinds of contraction (left, right, for vectors, multivectors, forms, tensors), some very abstract \cite{Bourbaki1998}.
Contraction by vectors is prevalent, even when multivectors could be of great use, as in \cite{Winitzki2010}.
Contraction between multivectors is simpler than with forms, but needs an inner product.
Hestenes inner product, of Clifford Geometric Algebra (GA) \cite{Dorst2007,Hestenes1984}, is a symmetrized contraction with worse properties.

Also, contractions are often presented in ways that obfuscate their simple nature as adjoints of exterior products.
For example, geometers view the contraction or insertion of vector fields on differential forms as an antiderivation linked to exterior and Lie derivatives (\cite[p.429]{Abraham1988}, \cite[p.207]{ChoquetBruhat1991}, \cite[p.35]{Kobayashi1996a}), 
	\CITE{Lang, Diff. Riem. Mflds. p.137}
making it seem more complicated than necessary.

Different conventions are another source of confusion: notations vary, and non\-e\-quiv\-a\-lent definitions give contractions with distinct properties.
For example, \cite{Bourbaki1998,Browne2012,Dorst2007,Gallier2020,Hestenes1984,Marcus1975,Pavan2017,Reichardt1957,Rosen2019,Sternberg1964}
	\CITE{BayroCorrochano2018, Shaw1983} 
have each a different contraction. 
Most authors present their favorite one without warning about the others, and their differences are barely discussed in the literature.
\Cref{sc:Appendix} fills this void, and those who are familiar with some contraction (including Hestenes product) may want to take a look at it first.

This work organizes, simplifies and extends the theory of contractions and related subjects.
Most results can be found throughout the literature, but usually only for simple or homogeneous elements, and in forms that seem at odds due to the various conventions. 
Here they are presented in a more general, uniform and streamlined way, with simpler proofs.
This is possible thanks to our notation, 
an improved multi-index formalism, a mirror principle,
and use of general multivectors right from the start.

New results include: 
geometric characterizations of contractions and regressive products; 
higher order graded Leibniz rules for contractions with exterior and Clifford products; 
determinant formulas; 
etc. 
We also study star operators akin to the Hodge star or the dual of GA \cite{Dorst2007,Rosen2019}, 
and a new involution simplifies their use.
A natural concept of complex orientation gives simpler stars, better suited for complex geometry \cite{Huybrechts2004,Wells1980}.

For simplicity, we use contractions of multivectors in Euclidean or Hermitian vector spaces, but most results adapt for contractions with forms, in pseudo-Euclidean spaces, or on manifolds.
The complex case, often neglected but important for geometry and quantum theory, differs from the real one in that contractions are sesquilinear.

Use of general multivectors whenever possible simplifies the theory.
It requires left and right contractions, but a mirror principle facilitates their use.
Authors who focus on homogeneous elements find it redundant to have both, as in such case they differ only by grade dependent signs.
But these signs clutter the algebra, force one to keep track of grades,
and make it hard to work with non-homogeneous elements.
These are important since they 
result from Clifford products \cite{Hestenes1984};
represent quantum superpositions of states with variable numbers of fermions \cite{Oziewicz1986};
appear in Graph Theory \cite{Caracciolo2007grassmann}, via Berezin calculus (whose derivatives and integrals are indeed contractions \cite{Lasenby1993});
and can store data about sets of subspaces of mixed dimensions, which have many applications \cite{Draper2014,Mandolesi_TotalGrassmannian}.

In a follow-up paper \cite{Mandolesi_Contractions2}, 
we use contractions to study 
subspaces associated to a general multivector,
special factorizations and decompositions, 
new simplicity criteria and Plücker-like relations,
supercommutators of multi-fermion creation and annihilation operators, etc.

\Cref{sc:Preliminaries} sets up notation and concepts we use.
\Cref{sc:Contractions} defines contractions and studies their properties.
\Cref{sc:Star Operators and Duals} describes star operators and the regressive product.
\Cref{sc:Appendix} discusses different contractions and conventions found in the literature, in special those of GA.

\section{Preliminaries}\label{sc:Preliminaries}

\begin{table}[]
	\scriptsize
	\centering
	\renewcommand{\arraystretch}{1}
	\begin{tabular}{lll}
		\toprule
		Symbol & Description & Page
		\\
		\cmidrule(lr){1-1} \cmidrule(lr){2-2} \cmidrule(lr){3-3} 
		$X$ & $n$-dimensional Euclidean or Hermitian space & \pageref{df:X}
		\\
		$\inner{\cdot,\cdot}$ & Inner or Hermitian product in $X$ or $\bigwedge X$ & \pageref{df:inner}, \pageref{df:inner AB} 
		\\
		$\MM$, $\MM^q$, $\MM^q_p$, $\II$, $\II^q$, $\II^q_p$ & Sets of non-repeating or increasing multi-indices & \pageref{df:multiindices} 
		\\
		$|\ii|$, $\|\ii\|$ & $|\ii|=p$ and $\|\ii\|=i_1+\cdots+i_p$ for $\ii=(i_1,\ldots,i_p)$ & \pageref{df:length norm i} 
		\\
		$\rr\backslash\ss$, $\rr\ss$, $\ii \cup \jj$, $\ii \cap \jj$, $\ii \triangle \jj$, $\ii'$ & Operations with multi-indices & \pageref{df:operations multiindices} 
		\\
		$\ord{\rr}$ & Ordered $\rr$ & \pageref{df:ordered r} 
		\\
		$\epsilon_\rr$ & Sign of the permutation that orders $\rr$ & \pageref{df:epsilon} 
		\\
		$\pairs{\rr}{\ss}$ & Number of pairs $(r,s) \in \rr\times\ss$ with $r>s$ & \pageref{df:|r>s|} 
		\\
		$\bigwedge V$, $\bigwedge^p V$, $\bigwedge^+ V$ & Exterior algebra, exterior power, even subalgebra & \pageref{df:exterior algebra}, \pageref{df:even subalgebra} 
		\\
		$\wedge$, $\vee$ & Exterior and regressive products & \pageref{df:wedge}, \pageref{df:regressive}
		\\
		$v_\rr$ & $v_\rr = v_{r_1}\wedge\cdots\wedge v_{r_p}$, $v_\emptyset = 1$ & \pageref{df:vr}  
		\\
		$\delta_{\mathbf{P}}$ & Propositional delta (1 if $\mathbf{P}$ is true, 0 otherwise)  & \pageref{df:delta} 
		\\
		$\comp{M}{p}$ & Component of grade $p$ of a multivector $M$ & \pageref{df:component} 
		\\
		$|H|$ & Grade of a homogeneous multivector $H$  & \pageref{df:grade} 
		\\
		$[B]$ & Space of a blade $B$ & \pageref{df:space} 
		\\
		$\|M\|$ & Norm of $M$ & \pageref{df:norm}
		\\
		$P_V$, $\PP_\VV$, $P_B$ & Orthogonal projections & \pageref{df:orthogonal projections} 
		\\
		$\hat{M}, \tilde{M}, \check{M}, \hhat{M}{k}$ & Involutions of $M$ & \pageref{df:involutions} 
		\\
		$\lcontr$, $\rcontr$ & Left and right contractions  & \pageref{df:contraction} 
		\\
		$\pperp$ & Partial orthogonality & \pageref{df:orthogonalities} 
		\\
		$B_P$, $B_\perp$ & Subblades of $B$ in a PO factorization & \pageref{df:PO factorization} 
		\\
		$\Theta_{V,W}$ & Asymmetric angle & \pageref{df:asymm angle} 
		\\
		$M\wedge \bigwedge V$ & $\{M\wedge N:N\in\bigwedge V\}$ & \pageref{df:wedge bigwedge} 
		\\
		$\adj{T}$ & Adjoint of a linear map $T$ & \pageref{df:adjoint} 
		\\
		$\Omega$ & Orientation (unit pseudoscalar) & \pageref{df:star duals}
		\\
		$\star$, $\star_L$, $\star_R$, $\star_B$ & Star operators & \pageref{df:star duals}, \pageref{df:star duals wrt B}
		\\
		$\lH{\! M}$, $\rH{M}$, $\lHB{\! M}$, $\rHB{M}$ & Left and right duals of $M$ & \pageref{df:star duals}, \pageref{df:star duals wrt B} 
		\\
		\bottomrule
	\end{tabular}
	\caption{Some symbols used in this article}
	\label{tab:symbols}
\end{table}

In this article, $X$\label{df:X}
is an $n$-dimensional Euclidean or Hermitian space, with inner product $\inner{\cdot,\cdot}$\label{df:inner} 
(Hermitian product in the complex case, conjugate-linear in the left entry).
When we mention linearity,  it is to be understood in the complex case as sesquilinearity, if appropriate.

\subsection{Multi-index formalism}\label{sc:Multi-index Notation}

For $1\leq p\leq q$,
let\label{df:multiindices}
$\II_p^q= \{ (i_1,\ldots,i_p)\in\N^p : 1\leq i_1 < \cdots<i_p\leq q \}$
and
$\MM_p^q= \{ (i_1,\ldots,i_p)\in\N^p : 1\leq i_j \leq q,
i_j\neq i_k \text{ if } j\neq k \}$.
Also, let 
$\II_0^q=\MM_0^q=\{\emptyset\}$,
$\II^q = \bigcup_{p=0}^q \II_p^q$,
$\MM^q = \bigcup_{p=0}^q \MM_p^q$,
$\II=\bigcup_{q=0}^{\infty} \II^q$ and
$\MM=\bigcup_{q=0}^{\infty} \MM^q$.
Let $|\ii|=p$ and $\|\ii\|=i_1+\cdots+i_p$ for  $\ii=(i_1,\ldots,i_p)$,\label{df:length norm i}
and $|\emptyset|=\|\emptyset\|=0$. 
\CITE{Some authors refer to our norm $\|\ii\|$ as length $|\ii|$.}
We also write $(i_1,\ldots,i_p)$ as $i_1\cdots i_p$,
and, in general, use $\ii,\jj,\kk$ for elements of $\II$, and $\rr,\ss,\tt$ for those of $\MM$.

For $\rr,\ss\in\MM$,\label{df:operations multiindices}
form $\rr\backslash\ss \in \MM$ by removing from $\rr$ any indices of $\ss$.
If they are disjoint (no common indices), $\rr\ss \in \MM$ equals $\rr$ followed by $\ss$.
We write $\rr\subset \ss$ if all indices of $\rr$ are in $\ss$.
Ordering $\rr$ we form $\ord{\rr}\in\II$,\label{df:ordered r}
and $\epsilon_\rr$\label{df:epsilon} is the sign of the permutation that orders it 
($\epsilon_\emptyset=1$).
The number of pairs $(r,s) \in \rr\times\ss$ with $r>s$ is $\pairs{\rr}{\ss}$.\label{df:|r>s|}
\SELF{$|\rr\ss| = |\rr| + |\ss|$ \\ $\|\rr\ss\| = \|\rr\| + \|\ss\|$ \\ 
	$\pairs{\rr\ss}{\tt} = \pairs{\rr}{\tt} + \pairs{\ss}{\tt}$ \\ $\pairs{\rr}{\ss\tt} = \pairs{\rr}{\ss} + \pairs{\rr}{\tt}$ \\
	$\pairs{\ord{\rr}}{\ss} = \pairs{\rr}{\ord{\ss}}  = \pairs{\rr}{\ss}$ \\
	$\pairs{\rr}{\rr} = \frac{|\rr|(|\rr|-1)}{2}$}
For $\ii,\jj\in\II$, form $\ii \cup \jj$, $\ii \cap \jj$ and $\ii \triangle \jj \in \II$ by ordering their union, intersection and symmetric difference.
For $\ii\in\II^q$, let $\ii' = (1,\ldots,q)\backslash\ii$ (its dependence on $q$ is left implicit).

\begin{proposition}\label{pr:epsilon}
	Let $\rr, \ss \in\MM$ and $\ii,\jj,\kk \in\II$  be pairwise disjoint.
		\SELFL{$\epsilon_{\ii\rr} \cdot \epsilon_{\jj\rr} = \epsilon_{\ord{\ii\jj}\,\ord{\rr}}$ \\ $\epsilon_{\rr\ii} \cdot \epsilon_{\rr\jj} = \epsilon_{\ord{\rr}\,\ord{\ii\jj}}$}
	\begin{enumerate}[i)]
		\item $\epsilon_{\rr \ss} = (-1)^{|\rr||\ss|} \epsilon_{\ss \rr}$. \label{it:epsilon graded commutation adjacent}
			
		\item $\epsilon_{\rr \ss} = \epsilon_{\rr} \, \epsilon_{\ord{\rr}\ss}$. 
		\label{it:epsilon ord ra}
		
		\item $\epsilon_{\ii\jj} = (-1)^{\pairs{\ii}{\jj}}$. \label{it:epsilon pairs}
			\CITE{[p.\,32]{Rosen2019}}
			\SELF{same parity as number of $i\in\ii$ for which $\{j\in\jj:j<i\}$ has odd cardinality}
					
		\item $\epsilon_{\ii\ii'} =  (-1)^{\frac{|\ii|(|\ii|+1)}{2} + \|\ii\|}$.\label{it:epsilon i}
			\SELF{$\epsilon_{\ii\ii'}$ does not depend on $\II^q$ used to take $\ii'$, but $\epsilon_{\ii'\ii}$ does}

		\item $\epsilon_{\ii\jj\kk} = \epsilon_{\ii \jj} \, \epsilon_{\ii\kk} \, \epsilon_{\jj\kk}$. \label{it:epsilon r1 rm}
			\SELF{$\epsilon_{\rr_1 \cdots \rr_m} = \prod\limits_{1\leq a < b \leq m} \epsilon_{\rr_a \rr_b}$ \  if $\rr_i\in\II$, or $\rr_i\in\MM$ and $m$ is even}
			\SELF{$\epsilon_{\ord{\ii_1\cdots\ii_m}\,\ord{\jj_1\cdots\jj_n}} = \prod\limits_{1\leq a\leq m, 1\leq b\leq n} \epsilon_{\ii_a \jj_b}$}
	\end{enumerate}
\end{proposition}
\begin{proof}
	\emph{(\ref{it:epsilon graded commutation adjacent})} It takes $|\rr||\ss|$ index transpositions to swap $\rr$ and $\ss$.
	\emph{(\ref{it:epsilon ord ra})} To order $\rr \ss$, we can order $\rr$, then order $\ord{\rr} \ss$.
	\emph{(\ref{it:epsilon pairs})} As $\ii$ and $\jj$ are ordered, the number of transpositions to put each $j\in\jj$ (first $j_1$, then $j_2$, etc.) in order in $\ii\jj$ is the number of indices $i\in\ii$ with $i>j$. 
	\emph{(\ref{it:epsilon i})} For an ordered $\ii=(i_1,\ldots,i_p)$, $\pairs{\ii}{\ii'} = (i_1-1)+(i_2-2)+\cdots+(i_p-p) = \|\ii\|-\frac{p(p+1)}{2}$. 
		\OMIT{\ref{it:epsilon pairs}} 
	\emph{(\ref{it:epsilon r1 rm})} $\epsilon_{\ii\jj\kk} = \epsilon_{\ii \jj} \,\epsilon_{\ord{\ii\jj}\kk}$ and $\epsilon_{\ord{\ii\jj}\kk} = (-1)^{\pairs{\ii\jj}{\kk}} = (-1)^{\pairs{\ii}{\kk} + \pairs{\jj}{\kk}} = \epsilon_{\ii\kk}\, \epsilon_{\jj\kk}$.
		\OMIT{\ref{it:epsilon ord ra}, \ref{it:epsilon pairs}}
\end{proof}


\subsection{Grassmann algebra}\label{sc:Grassmann Algebra}

\emph{Grassmann's exterior algebra} \cite{Bourbaki1998,Dorst2007,Rosen2019} of a subspace $V\subset X$ is a graded algebra $\bigwedge V=\bigoplus_{p\in\Z} \bigwedge^p V$,\label{df:exterior algebra}
with $\bigwedge^0 V = \{$scalars$\} = \R$ or $\C$, $\bigwedge^1 V = V$, and $\bigwedge^p V = \{0\}$ 
if $p \notin [0,\dim V]$.
Its bilinear associative \emph{exterior product} $\wedge$\label{df:wedge}
is alternating, with 
$A\wedge B = (-1)^{pq} B\wedge A \in \bigwedge^{p+q} V$
for $A\in \bigwedge^p V$ and $B\in \bigwedge^q V$.
For $u,v\in V$, $u \wedge v = -v\wedge u$, so $v\wedge v=0$.

\begin{definition}
	Given $v_1,\ldots,v_q \in X$ and $\rr=(r_1,\ldots,r_p)\in\MM_p^q$, let $v_\rr = v_{r_1}\wedge\cdots\wedge v_{r_p}$,\label{df:vr}
	and $v_\emptyset = 1$.
\end{definition}

\begin{definition}\label{df:delta}
	For a proposition $\mathbf{P}$, let $\delta_{\mathbf{P}} = 
	\begin{cases}
		1 \quad &\text{if $\mathbf{P}$ is true}, \\
		0 &\text{otherwise.}
	\end{cases}$ 
\end{definition}

We have  $v_\rr \wedge v_\ss = \delta_{\rr\cap \ss=\emptyset} \, v_{\rr\ss} = \delta_{\rr\cap \ss=\emptyset} \, \epsilon_{\rr\ss} v_{\ord{\rr\ss}}$ for $\rr,\ss\in\MM^q$, and so $v_\ii \wedge v_\jj = \delta_{\ii\cap \jj=\emptyset} \, \epsilon_{\ii\jj} \, v_{\ii \cup \jj}$
for $\ii,\jj\in\II^q$.
A basis $\beta_V=(v_1,\ldots,v_q)$ of $V$ gives bases
$\beta_{\bigwedge^p V} = \{v_\ii\}_{\ii\in\II_p^q}$ of $\bigwedge^p V$,
and
$\beta_{\bigwedge V} = \{v_\ii\}_{\ii\in\II^q}$ of $\bigwedge V$.
For $V=\{0\}$, $\beta_{\bigwedge V} = \beta_{\bigwedge^0 V} = \{1\}$.
If $U\subset V$ then $\bigwedge U \subset \bigwedge V$.

\begin{example}
	If $\beta_V=(v_1,v_2,v_3)$ then $\beta_{\bigwedge^2 V} = (v_{12},v_{13},v_{23})$ and $\beta_{\bigwedge V} = (1,v_1,v_2,v_3,v_{12},v_{13},v_{23},v_{123})$, for $v_{12}=v_1\wedge v_2$, $v_{13}=v_1\wedge v_3$, $v_{23}=v_2\wedge v_3$ and $v_{123}= v_1\wedge v_2\wedge v_3$.
	Also, $v_2 \wedge v_{13} = v_{213} = - v_{123}$ and $v_{13}\wedge v_{23}=0$.
\end{example}

Any $M\in\bigwedge X$ is a \emph{multivector}, 
and $\comp{M}{p}$\label{df:component}
is its component in $\bigwedge^p X$.
Any $H\in\bigwedge^p X$ is \emph{homogeneous} of \emph{grade} $\grade{H}=p$,\label{df:grade}
or a \emph{$p$-vector}.
For $v_1,\ldots,v_p\in X$, $B = v_{1\cdots p}$ is a \emph{simple} $p$-vector, or \emph{$p$-blade}.
We have $B\neq 0 \Leftrightarrow v_1,\ldots,v_p$ are linearly independent, in which case its \emph{space} is $[B] = \Span\{v_1,\ldots,v_p\} = \{v\in X: v\wedge B=0\}$.\label{df:space}
A scalar $\lambda$ is a $0$-blade, with $[\lambda]=\{0\}$,
	\CITE{cite[p.45]{Dorst2007}}
and $0$ is a $p$-blade for all $p$.
For a $p$-dimensional subspace $V$ and a $p$-blade $B\neq 0$, $V=[B] \Leftrightarrow\bigwedge^p V =\Span\{B\}$.
A blade $A$ is a \emph{subblade} of $B$ if $[A] \subset [B]$.
They have \emph{same orientation} if $A=\lambda B$ for $\lambda>0$.
Any $M\in\bigwedge V$ has a (non-unique) \emph{blade decomposition} $M = \sum_i B_i$ for blades $B_i \in \bigwedge V$.

To help distinguish results that only hold for certain kinds of multivectors, we usually (but not always) use $L,M,N$ for general multivectors, $F,G,H$ for homogeneous ones, and $A,B,C$ for blades.

The inner product of  $A=v_1\wedge\cdots\wedge v_p$ and $ B=w_1\wedge\cdots\wedge w_p$ is $\inner{A,B} = \det\!\big(\inner{v_i,w_j}\big)$.\label{df:inner AB}
It is extended linearly, with distinct $\bigwedge^p X$'s being orthogonal and $\inner{\kappa,\lambda}=\bar{\kappa} \lambda$ for $\kappa,\lambda\in\bigwedge^0 X$, where $\bar{\kappa}$ is the complex conjugate.
If a basis $\beta_V$ is orthonormal, so are $\beta_{\bigwedge V}$ and $\beta_{\bigwedge^p V}$.
The norm of $M$ is $\|M\| = \sqrt{\inner{M,M}}$.\label{df:norm}
In the real case, $\|A\|$ is the $p$-dimensional volume of the parallelotope spanned by $v_1,\ldots,v_p$.
In the complex case, $\|A\|^2$ gives the $2p$-dimensional volume of that spanned by $v_1,\im v_1,\ldots,v_p, \im v_p$.

Any linear map $T:X\rightarrow Y$ extends to an \emph{outermorphism}, a linear $T:\bigwedge X \rightarrow \bigwedge Y$ with%
\footnote{By convention, maps take precedence over products: $TM\wedge TN$ means $(TM)\wedge (TN)$.}
$T(M\wedge N) = T M\wedge T N$ for $M,N\in\bigwedge X$, and $T(1)=1$.
	\SELF{Precisa para não ter $T=0$}
If $B$ is a $p$-blade, so is $TB$, and $[TB] = T([B])$ if $TB\neq 0$.
	\OMIT{If $B=v_1\wedge\cdots\wedge v_p$ and $TB\neq 0$, $[TB] = \Span\{Tv_1,\ldots,Tv_p\} = T(\Span\{v_1,\ldots,v_p\}) = T([B])$}
Also, $T(\bigwedge^p V) = \bigwedge^p (T(V))$ for $V\subset X$.
Note that a scalar $\lambda$ times the outermorphism of $T$ is not an outermorphism (so, it is not that of $\lambda T$). 
 
We use $P_V:X\rightarrow V$ and $\PP_{\VV}:\bigwedge X\rightarrow \VV$\label{df:orthogonal projections}
for orthogonal projections onto subspaces $V\subset X$ and $\VV\subset \bigwedge X$, 
and $P_B = P_{[B]}$ for a blade $B$.
As an outermorphism, $P_V = \PP_{\bigwedge V}$.
For $p$-blades $A$ and $B\neq 0$,
$P_B A = \frac{\inner{B,A}}{\|B\|^2} B$.\label{eq:P_B A}
	\OMIT{Uso em \Cref{sc:Geometric characterization}. \\
		$P_{[B]} A = \PP_{\bigwedge^p [B]} A = \PP_{\Span\{B\}} A = \frac{\inner{B,A}}{\|B\|^2} B$}
If $A \neq 0$ we have $P_V A \neq 0 \Leftrightarrow [P_V A] = P_V([A])$.\label{pr:P_V A neq 0}
	\OMIT{For $A=v_1\wedge\cdots\wedge v_p$, $P A = Pv_1\wedge\cdots\wedge Pv_p \neq 0 \Leftrightarrow [PA] =  \Span\{Pv_i$'s$\} = P[A]$. Scalar case trivial.}

\emph{Grade involution} $\hat{\ }$ and \emph{reversion} $\tilde{\ }$ are linear maps $\bigwedge X\rightarrow\bigwedge X$ given by
$\hat{M} = \sum_p (-1)^p \comp{M}{p}$
and
$\tilde{M} = \sum_p (-1)^\frac{p(p-1)}{2} \comp{M}{p}$.\label{df:involutions}
	\SELFL{Patterns are $+-+-$ and $++--$ for $p\!\mod 4 = 0,1,2,3$. \\
	In general, $\hat{M},\tilde{M} \neq \pm M$. \\ 
	Clifford conjugation $\bar{H} = (\tilde{H})\hat{\ } = (-1)^{\frac{p(p+1)}{2}} H$, pattern $+--+$} 
For $M, N\in\bigwedge X$, $(\hat{M})\hat{\ } = (\tilde{M})\tilde{\ } = M$, $\inner{\hat{M},\hat{N}} = \inner{\tilde{M},\tilde{N}} = \inner{M,N}$
	\OMIT{$\inner{M,N}=0$ for $M\in\bigwedge^p X$ and $N\in\bigwedge^q X$ with $p\neq q$}
and $\hhat{(M\wedge N)}{} = \hat{M}\wedge\hat{N}$,
but $\tilde{\ }$ reverses the order, 
\SELF{it is an anti-involution \wrt this product}
$(M\wedge N)\tilde{\ } = \tilde{N} \wedge \tilde{M}$.
\OMIT{assuming $M\in\bigwedge^p X$ and $N\in\bigwedge^q X$, $(M\wedge N)\tilde{\ } = (-1)^{\frac{(p+q)(p+q-1)}{2}} M\wedge N = (-1)^{pq} \tilde{M} \wedge \tilde{N} = \tilde{N} \wedge \tilde{M}$}

\begin{definition}\label{df:repeated grade involution}
	Let $\hhat{M}{k}$
	be a composition of $k$ grade involutions of $M$,
	and 
	$\check{M} = \hhat{M}{(n+1)}$ for $n=\dim X$.
	\SELFR{$\hhat{M}{k} = \begin{cases}
			\hat{M} \quad &\text{if } k \text{ is odd,}  \\
			M &\text{if } k \text{ is even,}
		\end{cases}$
		\\
		$\hhat{H}{k} = (-1)^{kp} H$ if $|H|=p$
		\\
		$\check{M} = \begin{cases}
			M \quad &\text{if } n \text{ is odd,}  \\
			\hat{M} &\text{if } n \text{ is even.}
		\end{cases}$} 
\end{definition}

So, $\check{M} = \sum_p (-1)^{p(n+1)} \comp{M}{p}$, 
and $H\wedge M = \hhat{M}{p}\wedge H$ if $H\in\bigwedge^p X$.
Note that $\check{M} = M$ if $n$ is odd, $M \in \bigwedge^n X$ or $M \in \bigwedge^+ X$, where $\bigwedge^+ X = \bigoplus_{k\in\Z}\bigwedge^{2k} X$ is the \emph{even subalgebra}.\label{df:even subalgebra}

\section{Contractions}\label{sc:Contractions}

We will consider contractions between multivectors, for simplicity, but most of the theory adapts for the other kinds discussed in \Cref{sc:Appendix}.
It also adapts for pseudo-Euclidean spaces \cite{Pavan2017,Rosen2019,Shaw1983}, but some results hold only for non-null blades or have signature-dependent signs:
e.g., \eqref{eq:contr induced basis} becomes $v_{\ii} \lcontr v_{\jj} = \delta_{\ii\subset\jj} \, \sigma_\ii \, \epsilon_{\ii(\jj\backslash\ii)} \, v_{\jj\backslash\ii}$, with $\sigma_\ii = \inner{v_\ii,v_\ii} = \pm 1$.
Contractions can also be defined in spaces with degenerate metrics,
via their basic operational properties (\cite[p.73]{Dorst2007}, \cite[p.\,223]{Lounesto1993}).
	\OMITR{via \ref{pr:main tools}\ref{it:wedge contractor}, \ref{pr:main tools}\ref{it:v lcontr ws}, scalar contraction, sesquilinearity}

\begin{definition}\label{df:contraction}
	The \emph{left contraction} $M\lcontr N$ of a \emph{contractor} $M\in\bigwedge X$ on a \emph{contractee} $N\in\bigwedge X$, and the \emph{right contraction} $N\rcontr M$ of $N$ by $M$, are the unique elements of $\bigwedge X$ satisfying, for all $L\in\bigwedge X$,
	\begin{equation}\label{eq: interior exterior adjoints}
		\inner{L,M\lcontr N} = \inner{M\wedge L,N} \quad \text{ and }\quad \inner{L,N\rcontr M} = \inner{L\wedge M,N}.
	\end{equation}
\end{definition}

So, $M\lcontr$ and $\rcontr M$ are the adjoint operators of $M\wedge$ and $\wedge M$, respectively.
Contractions are bilinear in the real case, but in the complex one they are conjugate-linear in the contractor and linear in the contractee.
We often prove results only for the left contraction, as the right one is similar.

\begin{proposition}\label{pr:contr induced basis}
	For an orthonormal basis $(v_1,\ldots,v_n)$ and $\ii,\jj\in\II^n$,
		\CITE{cite[p.\,364]{Shaw1983}, cite[p.2242]{Kozlov2000I}, cite[p.52]Rosen}
	\begin{equation}\label{eq:contr induced basis}
		\begin{aligned}
			v_{\ii} \lcontr v_{\jj} &= \delta_{\ii\subset\jj} \, \epsilon_{\ii(\jj\backslash\ii)} \, v_{\jj\backslash\ii}, \text{ and} \\ 
			v_{\jj} \rcontr v_{\ii} &= \delta_{\ii\subset\jj} \, \epsilon_{(\jj\backslash\ii)\ii} \, v_{\jj\backslash\ii}.
		\end{aligned}
	\end{equation}
\end{proposition}
\begin{proof}
	For $\kk\in\II^n$, $\inner{v_\kk,v_\ii\lcontr v_\jj} = \inner{v_\ii \wedge v_\kk,v_\jj}$ vanishes unless $\ii\cap\kk=\emptyset$ and $\jj=\ord{\ii\kk}$ (so $\ii\subset\jj$ and $\kk=\jj\backslash\ii$), in which case it gives $\epsilon_{\ii(\jj\backslash\ii)}$.
\end{proof}

So $v_\ii \lcontr v_\jj$ (\resp $v_\jj \rcontr v_\ii$) is $0$ if $\ii\not\subset\jj$, otherwise the elements of $v_\ii$ are moved to the left (\resp right) of $v_\jj$ and canceled.
For a scalar $\lambda$ and $M\in\bigwedge X$ we have $\lambda \lcontr M = \bar{\lambda} M$, and $M\lcontr \lambda = \bar{\kappa}\lambda$ with $\kappa = \comp{M}{0}$.

\begin{example}\label{ex:contr}
	If $M = v_4 + \im v_{234}$ and $N = (\im - 3) v_{34} + v_{124}$ for $(v_1,\ldots,v_4)$ orthonormal then 
	$M\lcontr N 
	= (\im-3) v_4\lcontr v_{34} + v_4 \lcontr v_{124} - \im(\im - 3)  v_{234}\lcontr v_{34} - \im v_{234}\lcontr v_{124} 
	= (3-\im) v_4\lcontr v_{43} + v_4 \lcontr v_{412} + 0 + 0
	= (3-\im) v_3 + v_{12}$.
	Likewise, $N\rcontr M = (\im-3) v_3 + v_{12}$ and $N\lcontr M = M \rcontr N = (1-3\im) v_2$.
\end{example}

\begin{proposition}\label{pr:contr homog}
	Let $G\in\bigwedge^p X$, $H\in\bigwedge^q X$ and $M, N \in \bigwedge X$.
	\begin{enumerate}[i)]
		\item $G\lcontr H,H\rcontr G \in\bigwedge^{q-p} X$.\label{it:grade contr}
		\item If $p=q$ then $G \lcontr H = H\rcontr G = \inner{G,H}$.\label{it:contr inner}
		\item If $p>q$ then $G\lcontr H = H\rcontr G = 0$. \label{it:larger contractee}
		\item $G\lcontr M = \hhat{M}{p} \rcontr \hat{G}$ and $M\lcontr H = H \rcontr \hhat{M}{(q+1)}$.\label{it:contr left right M}
			\OMIT{\ref{pr:epsilon}\ref{it:epsilon graded commutation adjacent}}
		\item If $N\in\bigwedge V$ for $V\subset X$ then $M\lcontr N\in\bigwedge V$.\label{pg:M contr N subset N}
		\item $\hhat{(M\lcontr N)}{} = \hat{M}\lcontr\hat{N}$, $(M\lcontr N)\check{\ } = \check{M}\lcontr\check{N}$, but $(M\lcontr N)\tilde{\ } = \tilde{N}\rcontr\tilde{M}$. \label{it:contr involutions}
			\SELF{Cliff conj: \\ $(M\lcontr N)\bar{\ } = \bar{N}\rcontr\bar{M}$}
	\end{enumerate}
\end{proposition}
\begin{proof}
	\emph{(\ref{it:grade contr}--\ref{it:larger contractee})} Follow from \eqref{eq:contr induced basis} and linearity, with $G=\sum_{\ii\in\II_p^n} \lambda_\ii v_\ii$ and $H=\sum_{\jj\in\II_q^n} \kappa_\jj v_\jj$ for scalars $\lambda_\ii$, $\kappa_\jj$.
	\emph{(\ref{pg:M contr N subset N})} Likewise, extending an orthonormal basis $(v_1,\ldots,v_r)$ of $V$ to $X$, so $M=\sum_{\ii\in\II^n} \lambda_\ii v_\ii$ and $N=\sum_{\jj\in\II^r} \kappa_\jj v_\jj$.
	\emph{(\ref{it:contr involutions})} $\inner{L,(M\lcontr N)\tilde{\ }} = \inner{\tilde{L},M\lcontr N} = \inner{M\wedge \tilde{L}, N} = \inner{(M\wedge \tilde{L})\tilde{\ }, \tilde{N}} = \inner{L\wedge \tilde{M}, \tilde{N}} = \inner{L, \tilde{N} \rcontr \tilde{M}}$.
	Likewise for $\hat{\ }$ and $\check{\,}$, but without swapping $L$ and $M$.
\end{proof}

So, contractions generalize inner products, giving multivectors instead of scalars if grades differ.
They vanish if the contractor has larger grade,
and this will make the asymmetry $M\lcontr N \neq N\lcontr M$ useful.
While \ref{it:contr left right M} gives $G\lcontr H = (-1)^{p(q+1)} H\rcontr G$,
in general $M\lcontr N \neq \pm N \rcontr M$.
In \ref{it:contr involutions}, contractor and contractee keep their roles as $\tilde{\ }$ swaps them and switches $\lcontr$ and $\rcontr$.

A \emph{mirror principle} follows by applying $\tilde{\ }$ to a formula (with $\wedge$, $\lcontr$, $\rcontr$, $\pm$ or sideless operators that commute with $\tilde{\ }$, like $\hat{\ }$), distributing over all terms, and renaming them:
if the formula is valid for generic elements,
so is its mirror version, with terms in reversed order, and $\lcontr$ and $\rcontr$ switched.
For example, $\tilde{\ }$ applied to \ref{it:wedge contractor} below gives
$\tilde{N} \rcontr (\tilde{M} \wedge \tilde {L}) = (\tilde{N} \rcontr \tilde{L}) \rcontr \tilde{M}$,
and relabeling $\tilde{L}$, $\tilde{M}$, $\tilde{N}$ as $L$, $M$, $N$ we find the mirror formula 
$N\rcontr(M\wedge L) = (N\rcontr L)\rcontr M$.
The principle would be more general if notations were designed for this: e.g., $\inner{\cdot,\cdot}$ would have to show which entry is conjugate-linear (not the worst idea).
It is best to allow a certain flexibility, with some elements keeping their order (learning which ones takes just a little practice).



The following are the main tools for operating with contractions.

\begin{proposition}\label{pr:main tools}
	Let $v,w_1,\ldots,w_q \in X$, $H\in\bigwedge^p X$ with $p\leq q$, and $L,M,N\in\bigwedge X$. Then:
	\begin{enumerate}[i)]
		\item $(L\wedge M)\lcontr N = M\lcontr (L\lcontr N)$.\label{it:wedge contractor}
			\SELF{Adjoint formula of the associativity of $\wedge$}\SELF{$G\lcontr (H\lcontr M) = (-1)^{pq} H\lcontr (G\lcontr M)$, for  $G\in\bigwedge^p X$, $H\in\bigwedge^q X$}
		
		\item $H\lcontr w_{1\cdots q} = \sum_{\ii\in\II_p^q} \epsilon_{\ii\ii'} \,\inner{H,w_\ii}\,w_{\ii'}$. \label{it:H contr ws}
			\CITE{Multilinear Algebra, Greub p.120 ex. 6: fórmula equiv. p/contração multivetor com forma, notação pior. \\
			Geom Alg, Chisolm eq. 95: fórmula semelhante pra $\glcontr$}
		
		\item $v \lcontr w_{1\cdots q } = \sum_{i=1}^q  (-1)^{i-1} \inner{v,w_i} w_{1\cdots i' \cdots q}$, where $i'$ means $i$ is absent.\label{it:v lcontr ws}
			
		\item $v\lcontr(M\wedge N) = (v\lcontr M)\wedge N + \hat{M}\wedge(v\lcontr N)$.\label{it:Leibniz vector grade inv}
	\end{enumerate}
\end{proposition}
\begin{proof}
	\emph{(\ref{it:wedge contractor})} Follows from \eqref{eq: interior exterior adjoints} and associativity of $\wedge$.
	\emph{(\ref{it:H contr ws})}
	Linearity lets us assume $H$ is a blade. For a ($q-p$)-blade $B$,
	Laplace determinant expansion gives
	$\inner{B,H\lcontr w_{1\cdots q}} = \inner{H\wedge B,w_{1\cdots q}} = \sum_{\ii\in\II_p^q} \epsilon_{\ii\ii'} \inner{H,w_\ii}\,\inner{B,w_{\ii'}}$.
	\emph{(\ref{it:v lcontr ws})} 
	Follows from \ref{it:H contr ws}.
	\OMIT{For $v_1=v$ and $v_2,\ldots,v_q  \in X$ we have
		$\inner{v_{2\cdots q },v \lcontr w_{1\cdots q }} = \inner{v_{1\cdots q },w_{1\cdots q }} = \det(\inner{v_i,w_j}) = \sum_{i=1}^q  (-1)^{i-1} \, \inner{v,w_i}\, \inner{v_{2\cdots q },w_{1\cdots i' \cdots q }}$, by Laplace expansion.}
	\emph{(\ref{it:Leibniz vector grade inv})} Follows from \ref{it:v lcontr ws}, as we can assume $M=w_{1\cdots p}$ and $N=w_{(p+1)\cdots (p+q)}$ for $w_1,\ldots,w_{p+q} \in X$. 
	\OMIT{$v\lcontr(M\wedge N) 
		= \sum_{i=1}^p (-1)^{i-1} \inner{v,w_i}  w_{1\cdots i' 	\cdots p} \wedge N 
		+ \sum_{i=p+1}^{p+q} (-1)^{i-1} \inner{v,w_i}\cdot 	M\wedge w_{(p+1)\cdots i' \cdots (p+q)} 	
		= (v\lcontr M)\wedge N + (-1)^p M\wedge(v\lcontr N)$}
\end{proof}

The reordering of $L$ and $M$ in \ref{it:wedge contractor} reflects the adjoint nature of contractions. 
Some conventions avoid it, but have other difficulties (see \Cref{sc:Other Conventions}).
The mirror of \ref{it:H contr ws} is $w_{1\cdots q}\rcontr H = \sum_{\ii\in\II_p^q} \epsilon_{\ii'\ii} \,\inner{H,w_{\ii}}\,w_{\ii'}$ (after some relabeling).
\SELF{mirror $\epsilon_{\ii\ii'} \leftrightarrow \epsilon_{\ii'\ii}$}
In \ref{it:v lcontr ws}, the sign is $+$ at the first term of the sum,
while $w_{1\cdots q } \rcontr v = \sum_{i=1}^q  (-1)^{q -i} \inner{v,w_i} w_{1\cdots i' \cdots q }$
has $+$ at the last one.
As \ref{it:Leibniz vector grade inv} is a graded Leibniz rule, $v \lcontr$ is a graded derivation.
The notation makes it, and $(M\wedge N)\rcontr v = M\wedge(N\rcontr v) + (M\rcontr v)\wedge \hat{N}$, look natural: as $v$ `approaches' $M\wedge N$ from either side, it applies $\hat{\ }$ on the term over which it `jumps'.
Some authors switch left and right contractions, losing this.
	\SELF{$(v\lcontr)^2=0$, so $(\bigwedge X, v\lcontr)$ is a differential graded algebra. \\ Koszul complex (Wikipedia).}

%

\begin{example}
	Let $v,w_1, \ldots, w_4 \in X$ and $H\in\bigwedge^2 X$. Then
	$v\lcontr (w_1\wedge w_2) = \inner{v,w_1}w_2 - \inner{v,w_2}w_1$, 
	while
	$(w_1\wedge w_2)\rcontr v = -\inner{v,w_1}w_2 + \inner{v,w_2}w_1$.
	Also,
	$H\lcontr w_{1234} =  \inner{H,w_{12}}w_{34}  - \inner{H,w_{13}}w_{24} + \inner{H,w_{14}}w_{23} + \inner{H,w_{23}}w_{14} - \inner{H,w_{24}}w_{13} + \inner{H,w_{34}}w_{12} = w_{1234} \rcontr H$.
		\SELFL{$H\lcontr w_{123} =  \inner{H,w_{12}}w_3  - \inner{H,w_{13}}w_2 + \inner{H,w_{23}}w_1 = w_{123} \rcontr H$ \\ For $H\in\bigwedge^3 X$:
		$H\lcontr w_{1234} =  \inner{H,w_{123}}w_4  - \inner{H,w_{124}}w_3 + \inner{H,w_{134}}w_2 - \inner{H,w_{234}}w_1 = -w_{1234} \rcontr H$}
\end{example}

\begin{corollary}\label{pr:v1...vk contr M}
	$v_{1\cdots p} \lcontr M = v_p\lcontr(\cdots\lcontr(v_1\lcontr M)\cdots)$, for $v_1,\ldots,v_p\in X$ and $M\in\bigwedge X$.
		\OMIT{\ref{pr:main tools}\ref{it:wedge contractor}}
\end{corollary}

\begin{corollary}\label{pr:v wedge (M contr N)}
	$v \wedge (M \lcontr N) = (M \rcontr v) \lcontr N + \hat{M} \lcontr (v \wedge N)$, for $v\in X$ and $M,N\in\bigwedge X$.
	\CITE{Hestenes1984 p.12 (1.43). \\ Uso em \cite{Mandolesi_Contractions2}}
\end{corollary}
\begin{proof}%
	\OMITL{\ref{pr:main tools}\ref{it:Leibniz vector grade inv}, and $M \leftrightarrow \hat{M}$}
	For $L \in \bigwedge X$, we have 
	$\inner{L, \hat{M} \lcontr (v \wedge N)} 
	= \inner{v \lcontr (\hat{M} \wedge L), N} 
	= \inner{(v \lcontr \hat{M}) \wedge L + M \wedge (v \lcontr L), N}
	= \inner{L, -(M \rcontr v) \lcontr N + v \wedge (M \lcontr N)}$.
\end{proof}

In terms of operators, this is the adjoint of \ref{it:Leibniz vector grade inv}, arranged for convenience.

\begin{corollary}\label{pr:complet orth Leibniz}
	Let $M\in \bigwedge V$ and $N\in\bigwedge W$.
	\begin{enumerate}[i)]
		\item If $L\in\bigwedge (W^\perp)$ then $L\lcontr(M\wedge N) = (L\lcontr M)\wedge N$. \label{it:contr L cperp N}
		\item If $H\in\bigwedge^p (V^\perp)$ then $H\lcontr(M\wedge N) =  \hhat{M}{p} \wedge(H\lcontr N)$. \label{it:contr H cperp M}
	\end{enumerate}
\end{corollary}
\begin{proof}
	\emph{(\ref{it:contr L cperp N})} Linearity and \Cref{pr:v1...vk contr M} let us assume $L=v \in W^\perp$,
		\OMIT{trivial if $L$ is scalar}
	in which case it follows as in the proof of \Cref{pr:main tools}\ref{it:Leibniz vector grade inv}.
		\OMIT{Não posso usar \ref{it:Leibniz vector grade inv} direto pois ainda não provei $v\lcontr N=0$}
	\emph{(\ref{it:contr H cperp M})} Likewise.
\end{proof}

In \cite{Mandolesi_Contractions2}, we show what the solutions of $v \wedge M = 0$ and $v \lcontr M=0$, with $v \in X$, reveal about the structure of a multivector $M \in \bigwedge X$.
For now, note that \ref{it:Leibniz vector grade inv} gives
	\OMITR{\ref{pr:contr homog}\ref{it:contr inner}}
$v \lcontr (v \wedge M) + v \wedge (v \lcontr M) = \|v\|^2 M$, so:

\begin{corollary}
	$v\wedge M = v\lcontr M =0 \Leftrightarrow v=0$ or $M=0$.
		\CITEL{[p.114]{Pavan2017} for $M$ homogeneous, with worse proof}
\end{corollary}

\begin{corollary}\label{pr:wedge contr ker image}
	For $0\neq v \in X$ and $M \in \bigwedge X$:
		\OMIT{\ref{pr:main tools}\ref{it:wedge contractor},\ref{it:Leibniz vector grade inv}}
	\begin{enumerate}[i)]
		\item $v\wedge M = 0 \Leftrightarrow M = v\wedge N$ for $N\in\bigwedge X$. In particular, $N = \frac{v \lcontr M}{\|v\|^2}$.\label{it:wedge ker image}
		\item $v\lcontr M = 0 \Leftrightarrow M = v\lcontr N$ for $N\in\bigwedge X$. In particular, $N = \frac{v\wedge M}{\|v\|^2}$.\label{it:contr ker image}%
			\SELFL{More generally, for $N = \frac{v+L}{\|v\|^2} \wedge M$ with $L \in \bigwedge ([v]^\perp)$}
	\end{enumerate}
\end{corollary}

\begin{corollary}\label{pr:equations isp subset osp}
	If $M\neq 0$, $\{v\in X: v\wedge M = 0\} \perp \{w\in X : w\lcontr M = 0 \}$.
\end{corollary}
\begin{proof}
	$v\wedge M = w \lcontr M = 0 \Rightarrow 0 = w \lcontr (v \wedge M) =	\inner{w,v} M \Rightarrow w \perp v$.
	\OMITR{\ref{pr:wedge contr ker image}, \ref{pr:main tools}\ref{it:Leibniz vector grade inv}, \ref{pr:contr homog}\ref{it:contr inner}}
\end{proof}

\begin{proposition}\label{pr: v contr M wedge N = 0}
	\SELF{Uso em \cite{Mandolesi_Contractions2}}
	For $v\in X$ and nonzero $M\in\bigwedge V$ and $N\in\bigwedge W$ in disjoint subspaces $V$ and $W$, $v\lcontr(M\wedge N) = 0 \Leftrightarrow v\lcontr M = v\lcontr N = 0$.
\end{proposition}
\begin{proof}%
	\OMIT{\ref{pr:contr homog}\ref{pg:M contr N subset N}}
	($\Leftarrow$) Follows from \Cref{pr:main tools}\ref{it:Leibniz vector grade inv}.
	($\Rightarrow$)
	If the largest grade in $M$ is $r\neq 0$,
	\OMIT{$r=0$ is trivial} 
	it is at most $r-1$ in $v\lcontr M  \in \bigwedge V$.
	Given bases $(v_1,\ldots,v_p)$ of $V$ and $(w_1,\ldots,w_q)$ of $W$,
	$(v\lcontr M)\wedge N$ has no component with $\ii\in\II^p_r$ in the basis $\{v_\ii \wedge w_\jj\}_{\ii\in\II^p\!,\, \jj\in\II^q}$ of $\bigwedge(V\oplus W)$.
	Unless $v\lcontr N = 0$, $\hat{M}\wedge(v\lcontr N)$ has, contradicting $(v\lcontr M)\wedge N + \hat{M}\wedge(v\lcontr N) = 0$.
	Likewise, $v\lcontr M = 0$.
\end{proof}

In \cite{Mandolesi_Contractions2}, we show how $M\lcontr N = 0$ is linked, in a sense, to orthogonality
(after all, contractions generalize inner products). 
For now, we have:

\begin{proposition}\label{pr:v contr M = 0}
	$v\lcontr M = 0 \Leftrightarrow M \in \bigwedge([v]^\perp)$, for $v \in X$ and $M \in \bigwedge X$.
\end{proposition}
\begin{proof}
	($\Rightarrow$)
	Assume $(v,w_1,\ldots,w_{n-1})$ is an orthonormal basis of $X$,
	so $\bigwedge X = \Span\{w_\ii, v\wedge w_\ii\}_{\ii\in\II^{n-1}}$. 
	By \Cref{pr:main tools}\ref{it:v lcontr ws}, $v \lcontr w_\ii = 0$ and $v\lcontr (v \wedge w_\ii) = w_\ii$.
		\OMIT{$v \lcontr w_\emptyset = 0$ trivially} 
	By \Cref{pr:wedge contr ker image}\ref{it:contr ker image}, $M \in \Span\{v \lcontr w_\ii,v\lcontr (v \wedge w_\ii)\}_{\ii\in\II^{n-1}} = \Span\{w_\ii\}_{\ii\in\II^{n-1}} = \bigwedge([v]^\perp)$.
	($\Leftarrow$) $M \in \Span\{w_\ii\}_{\ii\in\II^{n-1}}$, so $v\lcontr M=0$.
\end{proof}

\begin{corollary}
	$[B] = \{v\in X:v \lcontr B = 0\}^\perp$, for a blade $B$.
		\OMIT{\ref{pr:v contr M = 0}}
\end{corollary}

\begin{definition}\label{df:orthogonalities}
	$U$ is \emph{partially orthogonal} ($\pperp$) to $V$ if $V^\perp \cap U \neq \{0\}$.
\end{definition}


For a blade $B\neq 0$, $[B]\pperp V \Leftrightarrow P_V B = 0$.\label{pr:PWB=0}
	\OMIT{p. \pageref{pr:P_V A neq 0}}

\begin{proposition}\label{pr:B contr M pperp}
	Given a blade $B\neq 0$ and a subspace $V\subset X$, we have $[B] \pperp V \Leftrightarrow B \lcontr M = 0$ for all $M \in \bigwedge V$.
\end{proposition}
\begin{proof}
	($\Rightarrow$) Follows from Propositions \ref{pr:v1...vk contr M} and \ref{pr:v contr M = 0}.
	($\Leftarrow$) $\|P_V B\|^2 = \inner{B, P_V B} = B \lcontr P_V B = 0$.
\end{proof}

\begin{corollary}\label{pr:M contr N orthog}
	$M \in \bigwedge V$, $N \in \bigwedge (V^\perp) \Rightarrow M \lcontr N = \bar{\lambda} N$ for $\lambda = \comp{M}{0}$. 
		\OMIT{\ref{pr:B contr M pperp}, for a blade decomposition of $M$}
\end{corollary}

In such case, if $M$ has no scalar component then $M \lcontr N = 0$.


Next, we combine left and right contractions.

\begin{proposition}\label{pr:assoc lcontr rcontr}
	$L\lcontr (M\rcontr N) = (L\lcontr M)\rcontr N$, for $L, M, N \in \bigwedge X$. 
\end{proposition}
\begin{proof}
	Follows from \eqref{eq: interior exterior adjoints} and associativity of $\wedge$.
\end{proof}

This lets us write just $L \lcontr M \rcontr N$.
Note the order of $\lcontr$ and $\rcontr$, as, in general, $L\rcontr (M\lcontr N) \neq (L\rcontr M)\lcontr N$.

\begin{proposition}\label{pr:triple subblade}
	Let $A$ and $B$ be blades, and $M,N \in\bigwedge X$.
	\begin{enumerate}[i)]
		\item If $[A] \subset [B]$ then $(A\rcontr M)\lcontr B = (P_A M)\wedge(A\lcontr B)$. \label{it:AMB}
		\item If $N \in \bigwedge [B]$ then 
		$(M\rcontr N)\lcontr B = N\wedge(M\lcontr B)$. \label{it:MAB}
		\CITE{\cite[p.\,594]{Dorst2007} proves \ref{it:MAB} for blades via 3 inductions on grades.}
	\end{enumerate}
\end{proposition}
\begin{proof}
	Extend an orthonormal basis of $[A]$ to $[B]$ and then to  $(v_1,\ldots,v_n)$ of $X$,
	and assume $N = A = v_\ii$, $B=v_\jj$ and $M=v_\kk$ for $\ii,\jj,\kk\in\II^n$ with $\ii\subset \jj$.
	\emph{(\ref{it:AMB})}  
	\SELF{If $\ii\subset\jj$: \\
		$(v_\ii \rcontr v_\kk) \lcontr v_\jj = v_\kk \wedge (v_\ii \lcontr v_\jj) \cdot \delta_{\kk \subset \ii}$ \\
		$(v_\kk \rcontr v_\ii) \lcontr v_\jj = v_\ii \wedge (v_\kk \lcontr v_\jj)$}
	If $\kk\not\subset \ii$ both sides vanish, otherwise $\ii=\ord{\kk\ll}$ and $\jj=\ord{\kk\ll\mm}$ for $\ll,\mm\in\II^n$, 
	and so 
	\OMIT{\ref{pr:epsilon}\ref{it:epsilon ord ra}, \ref{pr:contr induced basis}}
	$(v_\ii \rcontr v_\kk) \lcontr v_\jj 
	= \epsilon_{\ll\kk} v_\ll \lcontr v_\jj 
	= \epsilon_{\ll\kk} \epsilon_{\ll\kk\mm} v_{\kk\mm} 
	= \epsilon_{\ord{\kk\ll}\mm} v_{\kk} \wedge v_{\mm} 
	= v_\kk \wedge (v_\ii \lcontr v_\jj) 
	= (P_{v_\ii} v_\kk) \wedge (v_\ii \lcontr v_\jj)$.
	\emph{(\ref{it:MAB})} 
	Both sides vanish unless $\ii\subset \kk\subset \jj$, in which case \ref{it:AMB} gives $(v_\kk \rcontr v_\ii) \lcontr v_\jj = (P_{v_\kk} v_\ii) \wedge (v_\kk \lcontr v_\jj) = v_\ii \wedge (v_\kk \lcontr v_\jj)$.
\end{proof}

In \cite{Mandolesi_Contractions2}, we show how to replace $B$ by a general multivector. 
\CITE{pr:reformulated triple prods}
By \ref{it:AMB} and its mirror, $(B\rcontr M)\lcontr B = B\rcontr( M\lcontr B)$, so we can write $B\rcontr M\lcontr B$.

\begin{corollary}\label{pr:P B}
	$B\rcontr M\lcontr B = P_B M$, for a unit blade $B$ and $M\in\bigwedge X$.%
	\OMIT{\ref{pr:triple subblade}\ref{it:AMB} and its mirror}%
	\CITE{{Dorst2007} proves GA formula geometrically p.85, and algebraically p.106 \\
		Rosen exerc 2.8.9}
\end{corollary}

\begin{corollary}
	If $M = N \lcontr B$ for $N \in \bigwedge X$ and a blade $B \neq 0$ then $M =  \frac{B\rcontr M}{\|B\|^2} \lcontr B$. 
		\OMIT{\ref{pr:P B}}
		\MAYBE{estudar operadores anti-lineares $l_B(N)=N\lcontr B$ \\ $r_B(N)=B\rcontr N$?}
\end{corollary}

The method used in \Cref{pr:triple subblade}
	\SELFL{$v_\ii \rcontr v_\jj \lcontr v_\ii = v_\jj \cdot \delta_{\jj \subset \ii}$, \\
	$v_\ii \lcontr (v_\ii \wedge v_\jj) = v_\jj \cdot \delta_{\ii\subset\jj'} = v_\jj \cdot \delta_{\ii\cap\jj = \emptyset}$, \\
	$v_\ii \wedge (v_\ii \lcontr v_\jj) = v_\jj \cdot \delta_{\ii \subset \jj}$, \\
	$(v_\ii \lcontr v_\jj) \lcontr v_\ii = (v_\jj \lcontr v_\ii) \wedge v_\ii = v_\jj \cdot \delta_{\ii=\jj}$}
also gives other triple products.
For any blade $B$, we find $B\lcontr(M\lcontr B) = \bar{\lambda} \|B\|^2$ for $\lambda = \comp{M}{0}$,
	\OMIT{\ref{pr:contr homog}\ref{it:contr left right M}, \ref{pr:P B}}
and
$(B\lcontr M)\lcontr B = (M\lcontr B)\wedge B = \inner{M,B} B$ ($= \comp{P_B M}{p}$ if $B$ is a real unit $p$-blade).%
	\SELF{$= \PP_{\bigwedge^p [B]} M$.  Now all terms are $0$ if $\ii\neq\jj$. In complex case is conjugate-linear in $M$, so not a projection. These products coincide  only due to the $B$'s}
If it is nonscalar, $B\lcontr M \rcontr B = 0$.
\SELF{$= \bar{\lambda}^2 M$ if $B = \lambda$; \\
	$=(B\wedge B)\lcontr M = B\lcontr (B\lcontr M)$ \\
	\ \\
	$(B\wedge M)\lcontr B = M\lcontr (B\lcontr B)$ \\ $B\lcontr(M\lcontr B)$ não ganha sinal pois só fica termo escalar}

\subsection{Geometric interpretation}\label{sc:Geometric characterization}

Now we obtain complete geometric characterizations for contractions of nonzero blades $A\in\bigwedge^p X$ and $B\in\bigwedge^q X$.

\begin{definition}\label{df:PO factorization}
	$B=B_P\wedge B_\perp$ is a \emph{projective-orthogonal (PO) factorization} \wrt $A$ if 
	$B_P$ and $B_\perp$ are subblades of $B$ of grades $m=\min\{p,q\}$ and $q-m$, respectively, with $[B_\perp]$ orthogonal to $[A]$ and $[B_P]$.
		\SELFL{$B_\perp \in \bigwedge^{q-m}([A]^\perp \cap [B])$ \\ $B_P \in \bigwedge^m ([B_\perp]^\perp \cap [B])$ \\ Extende a $A=0$ ($p$ arbitrário, quaisquer $B_P, B_\perp$ servem) ou $B=0$ ($B_P=1$, $B_\perp=0$)}
\end{definition}


Note that
	\OMITR{$[A]\perp [B_\perp]$, and p. \pageref{eq:P_B A}}
$P_B A = P_{B_P} A = \frac{\inner{B_P,A}}{\|B_P\|^2} B_P$ ($=0$ if $p>q=m$).
	\SELF{So $P_W(V)\subset[B_P]$.
	If $V\not\pperp W$ (implies $p\leq q$) then $P_W(V) = [B_P]$ and $[B_\perp] = V^\perp \cap W$.}

\begin{proposition}\label{pr:contr PO}
	$A \lcontr B = \inner{A,B_P} B_\perp$.
\end{proposition}
\begin{proof}
	As $[A]\perp [B_\perp]$, $A \lcontr (B_P\wedge B_\perp) = (A \lcontr B_P) \wedge B_\perp$.
		\OMIT{\ref{pr:complet orth Leibniz}\ref{it:contr L cperp N}}
	And $A \lcontr B_P = \inner{A,B_P}$ (immediate if $p = m$,
		\OMIT{\ref{pr:contr homog}\ref{it:contr inner}}
	and if $p > m$ both vanish). 
\end{proof}

So, $A\lcontr B$ takes an inner product of $A$ with a subblade of $B$ where it projects,
	\SELF{\ie $P_B [A] \subset [B_P]$} 
contracting this subblade and leaving another orthogonal to it.
Likewise, $B \rcontr A = \inner{A,B_P} B_\perp'$
for $B_\perp' = \hhat{(B_\perp)}{m}$ (so $B = B_\perp' \wedge B_P$).


\begin{definition}\label{df:asymm angle}
	$\Theta_{V,W} = \cos^{-1} \frac{\|P_B A\|}{\|A\|}$ is the \emph{asymmetric angle} of $V=[A]$ with $W=[B]$.
\end{definition}

Formerly called Grassmann angle \cite{Mandolesi_Grassmann},
it is linked to the various products of GA \cite{Mandolesi_Products},
and gives an asymmetric Fubini-Study metric in the Total Grassmannian \cite{Mandolesi_TotalGrassmannian}.
As blade norms (squared, in the complex case) give volumes, $\cos \Theta_{V,W}$ (squared, in the complex case) is a \emph{projection factor} \cite{Mandolesi_Pythagorean} by which $p$-di\-men\-sion\-al ($2p$, in the complex case) volumes in $V$ contract if orthogonally projected on $W$.
We have $\Theta_{V,W} = 0 \Leftrightarrow V\subset W$, 
\SELF{$\Leftarrow$ não vale para $\Theta_{A,B}$, que também requer $A$, $B_P$ em fase}
$\Theta_{V,W} = \frac\pi2 \Leftrightarrow V \pperp W$,
\SELF{vale para $\Theta_{A,B}$}
and if $p=q$ then $\Theta_{V,W} = \Theta_{W,V} = \cos^{-1} \frac{|\inner{A,B}|}{\|A\|\|B\|}$.
	\SELF{and $\Theta_{A,B} = \overline{\Theta_{B,A}}$}
In general, $\Theta_{V,W} \neq \Theta_{W,V}$, reflecting a natural asymmetry between subspaces of different dimensions, linked to that of contractions.

\begin{theorem}\label{pr:characterization 1}
	For nonzero 
		\SELF{Vale trivially se $B=0$ e $p\neq 0$}
	blades $A\in\bigwedge^p X$ and $B\in\bigwedge^q X$, 
	we have
	$A \lcontr B = 0 \Leftrightarrow [A] \pperp [B]$.
	If $A\lcontr B \neq 0$,
	it is the only ($q-p$)-blade such that:
	\begin{enumerate}[i)]
		\item $[A\lcontr B] = [A]^\perp \cap [B]$. \label{it:space contr}
		\item $\| A\lcontr B\| = \|P_B A\|\| B\| = \| A\|\| B\|\cos\Theta_{[A],[B]}$. \label{it:norm}
		\item $(P_B A)\wedge (A\lcontr B)$ has the orientation of $B$.\label{it:compat orient}
	\end{enumerate}
\end{theorem}
\begin{proof}
		\OMIT{\ref{pr:contr PO}}
	$A\lcontr B = 0 \Leftrightarrow \inner{A,B_P} = 0 \Leftrightarrow P_B A = 0 \Leftrightarrow 
	[A] \pperp [B]$.
		\OMIT{as $A\neq 0$ (p. \pageref{pr:PWB=0})}
		\SELF{$[A] \perp [B] \Rightarrow [A]\pperp [B]$ if $p\neq 0$}
	Otherwise, 
	$[A\lcontr B] = [B_\perp] = [A]^\perp \cap [B]$,
	as $P_B([A]) = [P_B A] = [B_P]$,
	$\| A\lcontr B\| = |\inner{A,B_P}| \, \|B_\perp\| = \|P_B A\|\| B\|$,
	and 
	$(P_B A)\wedge (A\lcontr B) = \frac{|\inner{A,B_P}|^2}{\|B_P\|^2} B$. 
\end{proof}

The right contraction is similar, except that $(B\rcontr A)\wedge (P_B A)$ has the orientation of $B$
\SELF{\ref{it:space contr} allows a nice reformulation of the Gram-Schmidt process using contractions cite[p.\,65]{Rosen2019}.}
(Fig.\,\ref{fig:contraction 2}).
\SELF{Other PO factoriz switches signs of $B_P$, $B_\perp$ and $B_{\perp}'$}
Note that \ref{it:norm} holds for all blades, as $[A] \pperp [B] \Leftrightarrow \Theta_{[A],[B]} = \frac{\pi}{2}$.

\begin{figure}
	\centering
	\begin{subfigure}[b]{0.4\textwidth} 
		\includegraphics[width=\textwidth]{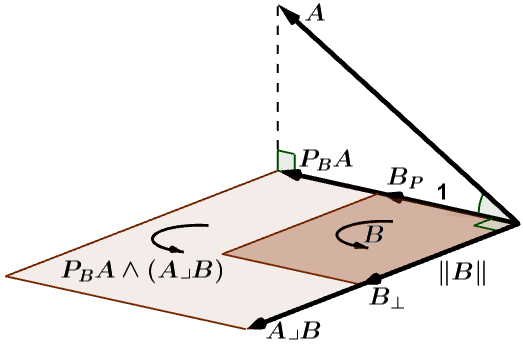}
		\caption{Left contraction $A\lcontr B$}
	\end{subfigure}
	\begin{subfigure}[b]{0.42\textwidth} 
		\includegraphics[width=\textwidth]{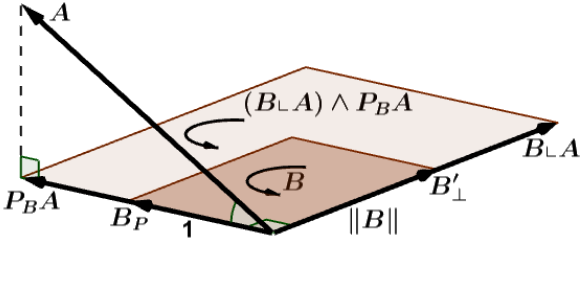}
		\caption{Right contraction $B \rcontr A$}
	\end{subfigure}
	\caption{Contractions remove from a 2-blade $B$ the direction where a vector $A$ projects, leaving a vector orthogonal to it.
	Rectangles are similar when $\|B_P\|=1$.}
	\label{fig:contraction 2}
\end{figure}

%
%
%

\begin{example}\label{ex:contr blade}
	If $A=v_2-2v_3+\im v_4$ and $B=v_{134}$ for an orthonormal basis $(v_1,\ldots,v_4)$ of $\C^4$ then
	$A \lcontr B 
	= 2 v_{14} - \im v_{13}$,
	so $[A]^\perp \cap [B] = [v_1]\oplus[2v_4-\im v_3]$.
	As
	$B_\perp = \frac{A\lcontr B}{\|A\lcontr B\|} = \frac{2v_{14}-\im v_{13}}{\sqrt{5}}$ 
	and $B_P = B \rcontr B_\perp = \frac{-2v_{3}+\im v_{4}}{\sqrt{5}}$
	give a PO factorization,
	$P_B A = -2v_3+\im v_4$
	and
	$(P_B A)\wedge (A\lcontr B) 
	= 5 B$.
	As 
	$\cos \Theta_{[A],[B]} = \frac{\sqrt{5}}{\sqrt{6}}$, 
	areas in the real plane $[A]$ contract by $\frac56$ if orthogonally projected on $[B]$ (each real dimension contracts by $\sqrt{\frac56}$).
	As 6-dimensional volumes of $[B]$ vanish if projected on $[A]$, $\Theta_{[B],[A]} = \frac\pi2$ and $B\lcontr A = 0$.
\end{example}

\subsection{Exterior and interior product operators}\label{sc:Operators}

Some properties are better expressed in terms of the following operators.

\begin{definition}\label{df:ext int products} 
	\emph{(Left) exterior and interior products by $M\in\bigwedge X$} are given, respectively, by $\ext_M (N) = M\wedge N$ and $\iota_M (N) = M\lcontr N$, for $N\in\bigwedge X$.%
	\CITE{A common notation for $\ext_M$ is $\varepsilon_M$, but it might cause confusion with our $\epsilon$'s.}
\end{definition}

Both are linear in $N$.
In $M$, $\ext_M$ is linear and $\iota_M$ is conjugate-linear.

\begin{proposition}
	$\ext_M\ext_N = \ext_{M\wedge N}$ and $\iota_M\iota_N = \iota_{N\wedge M}$, for $M,N \in \bigwedge X$,
\end{proposition}
\begin{proof}
	Follows from associativity of $\wedge$, and \Cref{pr:main tools}\ref{it:wedge contractor}.
\end{proof}

Note the order of $M$ and $N$ in $\iota_{N\wedge M}$.
If $M$ is odd, or a nonscalar blade, then $\ext_M^2 = \iota_M^2=0$, so $\Img \ext_M \subset \ker \ext_M$ and $\Img \iota_M \subset \ker \iota_M$.
	\MAYBE{Dá algo interessante do ponto de vista de holonomia?}

As $\ext_M$ and $\iota_M$ are adjoints, 
$\ker \ext_M = (\Img \iota_M)^\perp$ and $\ker \iota_M = (\Img \ext_M)^\perp$.
Also, $\iota_M \ext_M$ and $\ext_M \iota_M$ are self-adjoint.
	\SELF{and so is $\iota_M+\ext_M$. In real case, $\iota_v+\ext_v$ is left Clifford product by $v\in X$ and a Majorana operator in Physics}

\begin{proposition}
	Let $L,M,N \in \bigwedge X$.
	\begin{enumerate}[i)]
		\item $M\lcontr(M\wedge N)=0 \Leftrightarrow M \wedge N=0$.\label{it:ker iota ext}
		
		\item $M\wedge(M\lcontr N)=0 \Leftrightarrow M \lcontr N=0$.

		\item $L = M \lcontr N \Leftrightarrow L=M \lcontr (M \wedge K)$ for some $K\in \bigwedge X$. \label{it:Im iota ext}
		
		\item $L = M \wedge N \Leftrightarrow L=M \wedge (M \lcontr K)$ for some $K\in \bigwedge X$.
	\end{enumerate}
\end{proposition}
\begin{proof}
	Follows from usual properties of adjoint operators.
		\OMIT{\ref{it:ker iota ext} means $\ker(\iota_M \ext_M) = \ker \ext_M$. To prove, take $\inner{N,\cdot}$. 
		For \ref{it:Im iota ext}, $\Img(\iota_M \ext_M) = (\ker \adj{(\iota_M \ext_M)})^\perp = (\ker (\iota_M \ext_M))^\perp = (\ker \ext_M)^\perp = \Img \iota_M$. Others are similar.}	
\end{proof}

Let $M\wedge \bigwedge V = \{M\wedge N:N\in\bigwedge V\}$,
for $M\in\bigwedge X$ and $V\subset X$.\label{df:wedge bigwedge}

\begin{proposition}\label{pr:image contr}
	$\Img \ext_B = B \wedge \bigwedge([B]^\perp)$ 
	for any blade $B$, and if $B\neq 0$ then
	$\Img \iota_B = \bigwedge ([B]^\perp)$.
\end{proposition}
\begin{proof}
	Given orthonormal bases $(v_1,\ldots,v_p)$ of $[B]$ and $(w_1,\ldots,w_{n-p})$ 
	of $[B]^\perp$ ($p=0$ or $n$ is trivial), $\{v_\ii \wedge w_\jj\}_{\ii\in\II^p,\, \jj\in\II^{n-p}}$ is one for $\bigwedge X$.
	For $B=v_{1\cdots p}$ we have
	$\ext_B(v_\ii \wedge w_\jj) = \delta_{\ii=\emptyset} B\wedge w_\jj$
	and $\iota_B(v_\ii \wedge w_\jj) = \delta_{\ii=1\cdots p} w_\jj$.
\end{proof}

\begin{proposition}\label{pr:triple}
	Let $B$ be a unit blade, and $M\in\bigwedge X$.%
		\CITE{Rosen 2.8.7 dá \ref{it:P B perp} com $B$=vetor e dual, exerc 2.8.9 dá \ref{it:P B perp} com $M$=vetor}
		\SELF{$B$ unit blade: \\ $B\lcontr(B\wedge M) = M \Leftrightarrow M = B \lcontr N$ for $N \in \bigwedge X$ \\ $B\wedge(B\lcontr M) = M \Leftrightarrow M = B \wedge N$ for $N \in \bigwedge X$.}
	\begin{enumerate}[i)]
		\item $B\lcontr(B\wedge M) =\PP_{\Img \iota_B} M = P_{[B]^\perp} M$.\label{it:P B perp}%
		\SELFL{$= (M\wedge B)\rcontr B$; \\
			$(M\lcontr B)\lcontr B = P_B \hhat{M}{(p+1)}$}
		\item $B\wedge(B\lcontr M) = \PP_{\Img \ext_B} M$. \label{it:BBM}
	\end{enumerate}
\end{proposition}
\begin{proof}\OMITL{\ref{pr:image contr}}
	\emph{(\ref{it:P B perp})} 
	If $M \in \Img \iota_B = \bigwedge ([B]^\perp)$,
	\Cref{pr:complet orth Leibniz}\ref{it:contr L cperp N} gives $B\lcontr(B\wedge M) = M$.
	If $M \in (\Img \iota_B)^\perp = \ker \ext_B$ then
	$B\lcontr(B\wedge M) = 0$.
	%
	\emph{(\ref{it:BBM})} 
	If $M \in \Img \ext_B$ then $M = B \wedge N$ for $N \in
	\Img \iota_B$, 
	so $B\wedge(B\lcontr M) = B\wedge(B\lcontr (B \wedge N)) =  B \wedge N = M$.	
	If $M \in (\Img \ext_B)^\perp = \ker \iota_B$ then
	$B\wedge(B\lcontr M) = 0$.
\end{proof}

In \cite{Mandolesi_Contractions2}, $\ext_B$ and $\iota_B$ are linked to multi-fermion creation and annihilation operators,
and this result lets us interpret $\iota_B \ext_B$ and $\ext_B \iota_B$ as vacancy and occupancy operators, related to the quantum number operator.

Some cases of \ref{it:BBM} are worth mentioning.
For a unit $v\in X$ and $M\in\bigwedge X$, $v\wedge(v\lcontr M) = M-P_{[v]^\perp} M$,
	\CITE{Rosen 2.8.7 p/vetor e dual}
	\OMIT{\ref{pr:main tools}\ref{it:Leibniz vector grade inv}, \ref{pr:triple}\ref{it:P B perp}}
	\SELF{$=(\Id_{\bigwedge X} - P_{[v]^\perp})M \neq (\Id_X - P_{[v]^\perp})M$ (outermorphism functor not linear); \\	$\bigwedge X = (v \wedge \bigwedge X) \oplus \bigwedge ([v]^\perp)$}
and for $w \in X$, $v\wedge(v\lcontr w) = P_v w$.
	\OMIT{$v\wedge(v\lcontr w) = \inner{v,w} v = P_v w$}	
With a PO factorization 
	\SELFR{defini PO p/$A\neq 0$, estende}
$A_P\wedge A_\perp$ of a $q$-blade $A$ \wrt a unit $p$-blade $B$ (with $p\leq q$), 
one obtains $B\wedge(B\lcontr A) = P_{B\wedge A_\perp} A$ (Fig.\,\ref{fig:projections}).%
	\OMIT{$B\wedge(B\lcontr A) = \inner{B,A_P} B\wedge A_\perp = (P_B A_P) \wedge A_\perp = P_{B\wedge A_\perp} A$, \ref{pr:contr PO}}
	\SELF{$= \|A\|\cos(\Theta_{B,A})\, B\wedge A_\perp$ (oriented $\Theta$), if $p\leq q$ or $p>q\neq 0$. \\ 		$\inner{B,A_P} = \|A\|\cos\Theta_{B,A}$. \\ 		$B\wedge(B\lcontr A) = A$ if $B \subset A$}
The geometric relevance of $[B\wedge A_\perp]$ is 
that, among all $q$-dimensional subspaces 
$V\supset [B]$,
it is the closest one to $[A]$, in the sense that it minimizes $\Theta_{V,[A]}$ \cite{Mandolesi_Grassmann}.
	\CITE{3.8 e 3.9 da v6, 11/jan/21}

\begin{figure}[t]
	\centering
	\includegraphics[width=.5\textwidth]{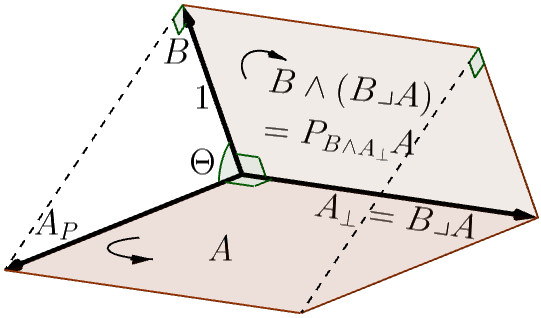}
	\caption{$B\wedge(B\lcontr A) = P_{B\wedge A_\perp} A$, if $\|B\|=1$. Among all planes containing $B$, that of $B\wedge A_\perp$ is the closest one to $A$ (it forms the smallest angle $\Theta$). The PO-factorization was chosen with $\|A_P\| = 1/\|P_A B\|$, so $\|A_\perp\|=\|B\lcontr A\|$.}
	\label{fig:projections}
\end{figure}


\begin{corollary}\label{pr:M=BBM}
	Let $B \neq 0$ be a blade, and $M\in\bigwedge X$.
	\OMIT{\ref{pr:triple}}
	\begin{enumerate}[i)]
		\item If $M = B \lcontr N$ for $N \in \bigwedge X$ then $M = B\lcontr \frac{B\wedge M}{\|B\|^2}$. \label{it:M = B contr B wedge M}
		\item If $M = B \wedge N$ for $N \in \bigwedge X$ then $M = B \wedge \frac{B \lcontr M}{\|B\|^2}$. \label{it:M = B wedge B contr M}
	\end{enumerate}
\end{corollary}

\begin{corollary}
	The restricted maps
	\begin{tikzcd}
		\bigwedge ([B]^\perp) \arrow[r,shift left=1,"\ext_B"] & B \wedge \bigwedge ([B]^\perp)  \arrow[l,shift left=1,"\iota_B"] 
	\end{tikzcd} 
	are mutually inverse isometries, for a unit blade $B$.
		\OMITL{By \ref{pr:image contr}, \ref{pr:triple}, mutually inverse adjoints, hence isometries}
\end{corollary}

Fig.\,\ref{fig:operators} shows how $e_B$ and $\iota_B$ act in $\bigwedge X$.

\begin{figure}
	\centering
	\includegraphics[width=0.3\linewidth]{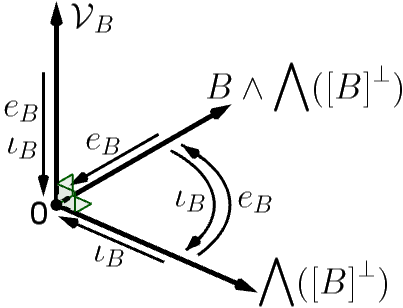}
	\caption{If $B$ is a unit nonscalar blade, $B \wedge \bigwedge([B]^\perp)$ and $\bigwedge([B]^\perp)$ are orthogonal, have orthogonal complement $\mathcal{V}_B = \sum_A \bigl(A \wedge \bigwedge([B]^\perp) \bigr)$, with a sum over nonscalar proper subblades $A$ of $B$, and $e_B$ and $\iota_B$ act as above, with curved arrows being inverses.}
	\label{fig:operators}
\end{figure}

For $0 \neq v \in X$, 
	\SELF{$\mathcal{V}_v = \{0\}$}
$\ker \ext_v = \Img \ext_v$ and $\ker \iota_v = \Img \iota_v$,
by \Cref{pr:wedge contr ker image},
so we have exact sequences
$0 \xrightarrow{\ext_v} \bigwedge\!^0 X \xrightarrow{\ext_v} \bigwedge\!^1 X \xrightarrow{\ext_v} \cdots  \xrightarrow{\ext_v} \bigwedge\!^n X \xrightarrow{\ext_v}  0$ 
and
$0 \xleftarrow{\iota_v} \bigwedge\!^0 X \xleftarrow{\iota_v} \bigwedge\!^1 X \xleftarrow{\iota_v}  \cdots  \xleftarrow{\iota_v} \bigwedge\!^n X  \xleftarrow{\iota_v}  0$.
	\SELFL{Self-duality of Koszul complexes says the two complexes are isomorphic. Isomorfismo seria dado pela star, by (pr:Hodge wedge contr)? Acho que teria que ser complexos do left exterior com a right contraction, e vice-versa. E se compor $\star$ com $\hat{\,}$\ ?}
If $\|v\|=1$, $\iota_v \ext_v + \ext_v \iota_v = \mathds{1}$.
	\OMIT{\ref{pr:main tools}\ref{it:Leibniz vector grade inv}, $v \lcontr (v \wedge M) = M - v \wedge (v \lcontr M)$}

\subsection{Higher order Leibniz rule}\label{sc:General Leibniz Rule}

Now we obtain higher order versions of the graded Leibniz rule and of its adjoint.
In \cite{Mandolesi_Contractions2} we interpret them in terms of supercommutators.

\begin{theorem}\label{pr:generalized Leibniz}
	For $v_1,\ldots,v_p\in X$ and $M,N\in\bigwedge X$,
	\CITER{cite[p.602]{Bourbaki1998}, remark (2): abstract formula for symmetric (?) algebras, in terms of coproduct $c(x)$. For anti-symmetric algebras, $c(B) = \sum_\ii \epsilon_{\ii\ii'} B_{\ii} \otimes B_{\ii'}$ (ex.7, p.576-577), so adding  grade involutions the formula could be adapted to give mine.}
	\SELF{Other formulation: \\
		$v_{1\cdots p}\lcontr(M\wedge N) = \sum_{\ii\in\II^p} \epsilon_{\ii\ii'}  \hhat{(v_\ii \lcontr M)}{|\ii'|} \wedge(v_{\ii'}\lcontr N)$}
	\begin{align}
		v_{1\cdots p}\lcontr(M\wedge N) &= \sum_{\ii\in\II^p} \epsilon_{\ii\ii'} (v_{\ii'}\lcontr \hhat{M}{|\ii|})\wedge(v_\ii\lcontr N), \text{ and} \label{eq:Leibniz generalizada} \\
		v_{1\cdots p} \wedge (M \lcontr N) &= \sum_{\ii\in\II^p} \epsilon_{\ii\ii'} (\hhat{M}{|\ii'|} \rcontr v_{\ii}) \lcontr (v_{\ii'} \wedge N). \label{eq:v1p wedge M contr N} 
	\end{align}
\end{theorem}
\begin{proof}
	\Cref{pr:main tools}\ref{it:Leibniz vector grade inv} gives \eqref{eq:Leibniz generalizada} for $p=1$ and, assuming it for $v_{1\cdots (p-1)}$, also
	\OMIT{\ref{pr:main tools}\ref{it:wedge contractor}}
	$v_{1\cdots p}\lcontr (M\wedge N) 
	= v_p\lcontr\big(v_{1\cdots (p-1)}\lcontr(M\wedge N)\big) 
	= v_p\lcontr \sum_{\ii\in\II^{p-1}} \epsilon_{\ii\ii'} (v_{\ii'}\lcontr \hhat{M}{|\ii|}) \wedge (v_\ii\lcontr N)
	= I + II$, 
	where
	\begin{align*}
		I &= \sum_{\ii\in\II^{p-1}} \epsilon_{\ii\ii'} \big(v_p\lcontr (v_{\ii'}\lcontr \hhat{M}{|\ii|})\big) \wedge (v_{\ii}\lcontr N)
		\\
		&= \sum_{\ii\in\II^{p-1}} \epsilon_{\ii\ii'} \big(v_{\ii' p}\lcontr\hhat{M}{|\ii|} \big) \wedge (v_{\ii}\lcontr N) 
		\\
		&= \sum_{\jj\in\II^p,\, p \notin \jj} \epsilon_{\jj\jj'} (v_{\jj'}\lcontr \hhat{M}{|\jj|}) \wedge (v_{\jj}\lcontr N),  
	\end{align*}
	for $\jj=\ii$ and $\jj'=\ii' p$,
	since $\epsilon_{\jj\jj'} = \epsilon_{\ii \ii' p} = \epsilon_{\ii\ii'}$,
	and 
	\begin{align*}
		II &= \sum_{\ii\in\II^{p-1}} \epsilon_{\ii\ii'} (v_{\ii'} \lcontr \hhat{M}{|\ii|})\hat{\,} \wedge \big(v_p\lcontr(v_{\ii}\lcontr N)\big) 
		\\
		&= \sum_{\ii\in\II^{p-1}} \epsilon_{\ii\ii'} \cdot (-1)^{|\ii'|} (v_{\ii'} \lcontr \hhat{M}{(|\ii|+1)}) \wedge (v_{\ii p}\lcontr N) 
		\\
		&= \sum_{\jj\in\II^p,\,p\in \jj} \epsilon_{\jj\jj'} (v_{\jj'}\lcontr \hhat{M}{|\jj|}) \wedge (v_{\jj}\lcontr N),
	\end{align*}
	for $\jj=\ii p$ and $\jj'=\ii'$, 
	since $\epsilon_{\jj\jj'} = \epsilon_{\ii p \ii'} = (-1)^{|\ii'|} \epsilon_{\ii\ii'}$. 
	
	\Cref{pr:v wedge (M contr N)} gives \eqref{eq:v1p wedge M contr N} for $p=1$ and, assuming it for $v_{1\cdots (p-1)}$, also
	$v_{1\cdots p} \wedge (M \lcontr N) 
	= v_p \wedge \hat{v}_{1\cdots (p-1)} \wedge (M \lcontr N) 
	= v_p \wedge \sum_{\ii\in\II^{p-1}} \epsilon_{\ii\ii'} (\hhat{M}{|\ii'|} \rcontr \hat{v}_{\ii}) \lcontr (\hat{v}_{\ii'} \wedge N)
	= I + II$, 
	where%
		\SELFR{$I = \sum_{\ii\in\II^{p-1}} \epsilon_{\ii\ii'} \cdot (-1)^{|\ii'|} (\hat{v}_{\ii p} \lcontr \hhat{M}{p}) \lcontr (v_{\ii'} \wedge N)$
		\\
		$= \sum_{\jj\in\II^p,\, p \in \jj} \epsilon_{\jj\jj'} (\hat{v}_{\jj} \lcontr \hhat{M}{p}) \lcontr (v_{\jj'} \wedge N)$,
		for $\jj=\ii p$ and $\jj'=\ii'$, since $\epsilon_{\jj\jj'} = \epsilon_{\ii p \ii'} = (-1)^{|\ii'|} \epsilon_{\ii\ii'}$, 
		\\
		and 
		\\ 
		$II = \sum_{\ii\in\II^{p-1}} \epsilon_{\ii\ii'} (\hat{v}_\ii \lcontr \hhat{M}{p}) \lcontr (v_{\ii' p} \wedge N)$
		\\
		$= \sum_{\jj\in\II^p,\, p \notin \jj} \epsilon_{\jj\jj'} (\hat{v}_\jj \lcontr \hhat{M}{p}) \lcontr (v_{\jj'} \wedge N),$
		for $\jj=\ii$ and $\jj'=\ii' p$,
		since $\epsilon_{\jj\jj'} = \epsilon_{\ii \ii' p} = \epsilon_{\ii\ii'}$.}
	\begin{align*}
		I &= \sum_{\ii\in\II^{p-1}} \epsilon_{\ii\ii'} \big((\hhat{M}{|\ii'|} \rcontr \hat{v}_{\ii}) \rcontr v_p \big) \lcontr (\hat{v}_{\ii'} \wedge N), \text{ and} 
		\\
		II &= \sum_{\ii\in\II^{p-1}} \epsilon_{\ii\ii'} (\hhat{M}{|\ii'|} \rcontr \hat{v}_{\ii})\hat{\,} \lcontr (v_p \wedge \hat{v}_{\ii'} \wedge N),
	\end{align*}
	which are then developed as above.
\end{proof}

The higher order graded Leibniz rule \eqref{eq:Leibniz generalizada} shows that contraction by a $p$-blade is a graded derivation of order $p$.
Writing in \eqref{eq:Leibniz generalizada} and \eqref{eq:v1p wedge M contr N} the terms for $\ii=\emptyset$ and $\ii=1\cdots p$, we find
\begin{align*}
	v_{1\cdots p}\lcontr(M\wedge N) &= (v_{1\cdots p} \lcontr M) \wedge N + \cdots + \hhat{M}{p} \wedge (v_{1\cdots p} \lcontr N), \text{ and}
	\\
	v_{1\cdots p} \wedge (M \lcontr N) &= \hhat{M}{p} \lcontr (v_{1\cdots p} \wedge N) + \cdots + (M \rcontr v_{1 \cdots p}) \lcontr N. 
\end{align*}
Though unobvious, \eqref{eq:v1p wedge M contr N} is equivalent to the adjoint of \eqref{eq:Leibniz generalizada},
$M \lcontr (v_{1\cdots p} \wedge N) = \sum_{\ii\in\II^p} \epsilon_{\ii\ii'}\, v_{\ii}\wedge \bigl((v_{\ii'} \lcontr \hhat{M}{|\ii|}) \lcontr N \bigr) = (v_{1\cdots p} \lcontr M) \lcontr N + \cdots + v_{1\cdots p} \wedge (\hhat{M}{p} \lcontr N)$,
and vice-versa.
	\SELF{$M \wedge (v_{1\cdots p} \lcontr N) = \sum_{\ii\in\II^p} \epsilon_{\ii\ii'}\, v_{\ii'} \lcontr \bigl((\hhat{M}{|\ii'|} \rcontr v_{\ii}) \wedge N \bigr) = v_{1\cdots p} \lcontr (\hhat{M}{p} \wedge N) + \cdots + (M \rcontr v_{1\cdots p}) \wedge N$}

\begin{example}\label{ex:generalized Leibniz}
	Let $v_1,v_2 \in X$ and $M,N\in\bigwedge X$. 
	By \eqref{eq:Leibniz generalizada},
	$v_{12}\lcontr(M\wedge N) 
	= (v_{12}\lcontr M)\wedge N  
	+ (v_2\lcontr \hat{M})\wedge(v_1\lcontr N) 
	- (v_1\lcontr \hat{M})\wedge(v_2\lcontr N)  
	+ M\wedge(v_{12}\lcontr N)$.
		\SELF{adjoint $M \lcontr (v_{12} \wedge N) = (v_{12} \lcontr M) \lcontr N + v_1 \wedge \big( (v_2 \lcontr \hat{M}) \lcontr N \big) - v_2 \wedge \big( (v_1 \lcontr \hat{M}) \lcontr N \big)	+ v_{12} \wedge (M \lcontr N)$}
	By \eqref{eq:v1p wedge M contr N},
	$v_{12} \wedge (M \lcontr N) 
	= M \lcontr (v_{12} \wedge N) 
	+ (\hat{M} \rcontr v_1) \lcontr (v_2 \wedge N)
	- (\hat{M} \rcontr v_2) \lcontr (v_1 \wedge N)
	+ (M \rcontr v_{12}) \lcontr N$.
		\SELF{adjoint $M \wedge (v_{12} \lcontr N) = v_{12} \lcontr (M \wedge N) + v_2 \lcontr \big( (\hat{M} \rcontr v_1) \wedge N \big)	- v_1 \lcontr \big( (\hat{M} \rcontr v_2) \wedge N \big) + (M \rcontr v_{12}) \wedge N$}
\end{example}

For more complex calculations one can use 
\SELFR{also $\hhat{v_\ii}{|\ii'|} = (-1)^{|\ii|\cdot(p+1)} v_\ii$}
\Cref{pr:epsilon}\ref{it:epsilon i},
noting that $(-1)^{\frac{|\ii|(|\ii|+1)}{2}}$ follows the pattern $+--+$ for $|\ii|\!\!\mod 4 = 0,1,2,3$,
\SELF{Same as Cliff conj.}
and $(-1)^{\|\ii\|} = (-1)^{\mathrm{odd}(\ii)}$, with $\mathrm{odd}(\ii) =$ number of odd indices in $\ii$.

\begin{example}\label{ex:generalized Leibniz 2}
	For $v_1,\ldots,v_4 \in X$ and $M,N\in\bigwedge X$, \eqref{eq:Leibniz generalizada} gives
	\begin{align*}
		v_{1234}\lcontr( M\wedge N) 
		= &+ (v_{1234}\lcontr M)\wedge N
		\\
		&- \bigl( 
		- (v_{234}\lcontr \hat{M})\wedge(v_1\lcontr N)
		+ (v_{134}\lcontr \hat{M})\wedge(v_2\lcontr N)
		\\&\phantom{- \bigl( }
		- (v_{124}\lcontr \hat{M})\wedge(v_3\lcontr N)
		+ (v_{123}\lcontr \hat{M})\wedge(v_4\lcontr N)
		\bigr)
		\\
		&- \bigl( 
		-(v_{34}\lcontr M)\wedge(v_{12}\lcontr N)
		+ (v_{24}\lcontr M)\wedge(v_{13}\lcontr N) 
		\\&\phantom{- \bigl( }
		- (v_{23}\lcontr M)\wedge(v_{14}\lcontr N)
		- (v_{14}\lcontr M)\wedge(v_{23}\lcontr N)
		\\&\phantom{- \bigl( }
		+ (v_{13}\lcontr M)\wedge(v_{24}\lcontr N)
		- (v_{12}\lcontr M)\wedge(v_{34}\lcontr N) 
		\bigr) 
		\\
		&+ \bigl( 
		+ (v_4\lcontr \hat{M})\wedge(v_{123}\lcontr N)
		- (v_3\lcontr \hat{M})\wedge(v_{124}\lcontr N)
		\\&\phantom{+ \bigl( }
		+ (v_2\lcontr \hat{M})\wedge(v_{134}\lcontr N) 
		- (v_1\lcontr \hat{M})\wedge(v_{234}\lcontr N)	
		\bigr) 
		\\
		&+ M\wedge(v_{1234}\lcontr N).
	\end{align*}
\end{example}

\subsection{Outermorphisms and contractions}%

Let $T:X\rightarrow Y$ be a linear map into a Euclidean or Hermitian space $Y$ (same as $X$).
The outermorphism of its adjoint $\adj{T}:Y\rightarrow X$\label{df:adjoint}
is the adjoint of its outermorphism, \ie $\inner{M,\adj{T}N} = \inner{TM,N}$ for $M\in\bigwedge X$ and $N\in\bigwedge Y$.
\CITE{[p.109]{Dorst2007}}
Likewise for its inverse, if it exists.
	\SELF{If $T$ is a vector space isomorphism, its outermorphism is an exterior algebra isomorphism (\wrt $\wedge$ only)}
We say $T$ is an \emph{isometry}
	\SELF{Não requer invertibilidade, mas implies injectivity}
if $\inner{Tx,Ty} = \inner{x,y}$ for all $x,y \in X$,
in which case $\adj{T} T = \Id$,
and $\inner{TM,TN} = \inner{M,N}$ for all $M,N \in \bigwedge X$.
By convention, maps take precedence over contractions: $TM\lcontr TN$ means $(TM)\lcontr (TN)$.

%

\begin{proposition}\label{pr:T contr adj}
	$T(\adj{T}M\lcontr N) = M\lcontr TN$, for $M\in\bigwedge Y$ and $N\in\bigwedge X$.\SELF{$\adj{T}(TN\lcontr M) = N\lcontr \adj{T}M$}\CITE{A similar result is given in \cite[Prop.\,2.7.1]{Rosen2019} in terms of adjoint maps between dual spaces.}
\end{proposition}
\begin{proof}
	For $L\in\bigwedge Y$, 
	$\inner{L,M\lcontr TN} = \inner{M \wedge L, TN} = \inner{\adj{T}(M \wedge L), N} = \inner{\adj{T} M \wedge \adj{T} L, N} = \inner{\adj{T} L, \adj{T} M \lcontr N} = \inner{L, T(\adj{T} M \lcontr N)}$.
\end{proof}



\begin{corollary}
	If $T$ is invertible then $T(M\lcontr N) =  (\adj{T})^{-1} M  \lcontr T N$, for $M,N\in\bigwedge X$.
	\OMIT{$M=\adj{T} L$ then \ref{pr:T contr adj}}
\end{corollary}

\begin{corollary}\label{pr:T contr}
	$T$ is an isometry $\Leftrightarrow T(M\lcontr N) = TM\lcontr TN$ for all $M,N\in\bigwedge X$. 
\end{corollary}
\begin{proof}
	($\Rightarrow$) $\adj{T} = T^{-1}$, so $M = \adj{T} L$ for $L= TM$.
	\OMIT{\ref{pr:T contr adj}}
	($\Leftarrow$) For $x,y\in X$, $\inner{Tx,Ty} = Tx \lcontr Ty = T(x\lcontr y) = T\inner{x,y} = \inner{x,y}$.
\end{proof}

\subsection{Determinant formulas}

For blades $A = v_1\wedge\cdots\wedge v_p$ and $B = w_1\wedge\cdots\wedge w_q$,
let $\mathbf{A}_{p \times q} = \big(\inner{v_i,w_j}\big)$,
$\mathbf{B}_{q \times q} = \big(\inner{w_i,w_j}\big)$ and
$\mathbf{M}_{(p+q)\times(p+q)}=\begin{psmallmatrix}
	\mathbf{0}  & \mathbf{A}  \\ 
	\mathbf{A}^\dagger & \mathbf{B}
\end{psmallmatrix}$,
where $^\dagger$ is the conjugate transpose.
Also, let $\mathbf{W}_{n \times q}$ have as columns the $w_j$'s decomposed in an arbitrary basis $(u_1,\ldots,u_n)$.%
\SELF{does not need orthogonality}
The following expansions can be useful when the determinants can be computed efficiently (\eg if they are sparse).

\begin{proposition}\label{pr:det expansion w}
	If $p\leq q$ then
	\begin{equation}\label{eq:det}
		A\lcontr B 
		= \sum_{\jj\in\II_p^q} \epsilon_{\jj\jj'} \det (\mathbf{A}_\jj) \,w_{\jj'}
		= \sum_{\ii\in\II_{q-p}^n}	\det
		\begin{psmallmatrix}
			\mathbf{A} \\
			\mathbf{W}_\ii
		\end{psmallmatrix}
		u_\ii,
	\end{equation}
	where  $\mathbf{A}_\jj$ 
	\OMIT{$p\times p$}
	is the submatrix of $\mathbf{A}$ formed by the columns with indices in $\jj$,
	and $\mathbf{W}_\ii$ is the submatrix of $\mathbf{W}$ formed by the lines with indices in $\ii$.
\end{proposition}
\begin{proof}
	The first equality is \Cref{pr:main tools}\ref{it:H contr ws}, as $\det \mathbf{A}_\jj = \inner{A,w_\jj}$.
	The other follows by Laplace expansion of
	$\det
	\begin{psmallmatrix}
		\mathbf{A} \\
		\mathbf{W}_\ii
	\end{psmallmatrix}$
	\wrt the first $p$ rows,
	as if $w_j = \sum_{i=1}^n \lambda_{ij} u_i$ then
	$w_{\jj'} = \sum_{\ii\in\II_{q-p}^n} \lambda_{\ii\jj'} u_\ii$ 
	for 
	$\lambda_{\ii\jj'} = \det(\lambda_{ij})_{i\in\ii,j\in\jj'}$.
\end{proof}


\begin{proposition}\label{pr:norm contr det}
	$\|A\lcontr B\| = \sqrt{|\det\mathbf{M}|}$.
\end{proposition}
\begin{proof}
	If $p>q$, Laplace expansion
	\CITE{cite{Muir2003}, p.80} 
	\wrt the first $p$ lines gives $\det \mathbf{M} = 0$.
	If $p\leq q$, 
	and $\mathbf{N}_{\jj'}= (\mathbf{A}^\dagger \ \, \mathbf{B}_{\jj'})$ 
	\OMIT{$q\times q$}
	is the submatrix of $\mathbf{M}$ formed by $\mathbf{A}^\dagger$ and the columns of $\mathbf{B}$ with indices not in $\jj$,
	the same expansion and \eqref{eq:det} give
	$\|A\lcontr B\|^2 = \inner{B,A\wedge(A\lcontr B)} = \sum_{\jj \in \II_p^q} \epsilon_{\jj\jj'} \det (\mathbf{A}_\jj)\, \inner{B,A\wedge w_{\jj'}} =
	\sum_{\jj \in \II_p^q} \epsilon_{\jj\jj'} \det \mathbf{A}_\jj \cdot \det \mathbf{N}_{\jj'} = (-1)^p \det \mathbf{M}$.
	\OMIT{$(-1)^p$ aparece pois, depois que passar a $p$ colunas $\jj$ para o início de $\mathbf{A}$, ainda tem que passá-las para antes do $\mathbf{0}$}
\end{proof}

If $B \neq 0$,
$\|A\lcontr B\| = \sqrt{\det \mathbf{B} \cdot \det(\mathbf{A} \mathbf{B}^{-1}\mathbf{A}^\dagger)}$,
by Schur's determinant identity.
\CITE{cite{Brualdi1983} \\ $\det(-\mathbf{A}\mathbf{B}^{-1}\mathbf{A}^\dagger) = (-1)^p \det(\mathbf{A}\mathbf{B}^{-1}\mathbf{A}^\dagger)$}
In \cite{Mandolesi_TotalGrassmannian}, we use it to compute an asymmetric Fubini-Study metric.

\begin{example}\label{ex:contr_det}\OMIT{Cálculos no arquivo Geogebra}
	Let $v_1,v_2,w_1,w_2,w_3 \in \R^3$ have
	$\mathbf{A} = 
	\left(\begin{smallmatrix}
		-1 & 0 & 1\\
		0 & -2 & 2
	\end{smallmatrix}\right)$, 
	and
	$\mathbf{W} = 
	\left(\begin{smallmatrix}
		1 & 1 & 0 \\
		0 & 2 & 1 \\
		1 & 0 & 3
	\end{smallmatrix}\right)$
	in a basis $\beta=(u_1,u_2,u_3)$.
	By \eqref{eq:det},
	$v_{12}\lcontr w_{123} = 
	\left|\begin{smallmatrix}
		-1 & 0 & 1\\
		0 & -2 & 2\\
		1 & 1 & 0
	\end{smallmatrix}\right| u_1 +
	\left|\begin{smallmatrix}
		-1 & 0 & 1 \\
		0 & -2 & 2 \\
		0 & 2 & 1
	\end{smallmatrix}\right| u_2 +
	\left|\begin{smallmatrix}
		-1 & 0 & 1 \\
		0 & -2 & 2 \\
		1 & 0 & 3
	\end{smallmatrix}\right| u_3
	= 4u_1 + 6u_2 + 8u_3$,
	same as expanding the $w$'s in
	$v_{12}\lcontr w_{123} = 
	\left|\begin{smallmatrix}
		-1 & 0  \\
		0 & -2  
	\end{smallmatrix}\right| w_3 -
	\left|\begin{smallmatrix}
		-1 & 1  \\
		0 & 2 
	\end{smallmatrix}\right| w_2 +
	\left|\begin{smallmatrix}
		0 & 1  \\
		-2 & 2 
	\end{smallmatrix}\right| w_1$.
	With $\mathbf{A} = (-1 \ 0 \ 1)$,
	$v_1\lcontr w_{123} = 
	\left|\begin{smallmatrix}
		-1 & 0 & 1 \\
		1 & 1 & 0 \\
		0 & 2 & 1
	\end{smallmatrix}\right| u_{12} +
	\left|\begin{smallmatrix}
		-1 & 0 & 1 \\
		1 & 1 & 0 \\
		1 & 0 & 3
	\end{smallmatrix}\right| u_{13} +
	\left|\begin{smallmatrix}
		-1 & 0 & 1 \\
		0 & 2 & 1 \\
		1 & 0 & 3
	\end{smallmatrix}\right| u_{23}
	= u_{12} - 4 u_{13} - 8 u_{23}$,
	same as
	$v_1\lcontr w_{123} = -1 w_{23} - 0 w_{13} + 1 w_{12}$.
	If $\beta$ is orthonormal, so $\mathbf{B} = \adj{\mathbf{W}} \mathbf{W}$, \Cref{pr:norm contr det} gives $\|v_{12}\lcontr w_{123}\| = \sqrt{116}$ and $\|v_1\lcontr w_{123}\|=9$.
\end{example}

\subsection{Clifford product and contractions}

In this section, $X$ is real%
\footnote{Complex Clifford algebras fail to reflect complex geometry: $vw \neq \inner{v,w} + v \wedge w$, since $vw$ is bilinear but $\inner{v,w}$ is sesquilinear. As complex numbers can be represented in real Clifford algebras, one can use these, but this is not always convenient.}.
Its \emph{Clifford Geometric Algebra} \cite{Dorst2007,Hestenes1984,Rosen2019} is $\bigwedge X$ with a bilinear associative \emph{Clifford geometric product} $MN$ for $M,N\in \bigwedge X$, which has $vv=\|v\|^2$ for $v\in X$, and $v_1v_2\cdots v_p = v_1 \wedge v_2 \wedge \cdots \wedge v_p$ for orthogonal $v_1,\ldots,v_p \in X$. 
So, $vw = \inner{v,w} + v \wedge w$ for $v,w \in X$.
For an orthonormal basis $(v_1,\ldots,v_n)$, as $v_i^2=1$ and $v_i v_j = -v_j v_i$ if $i\neq j$, we have
$v_\ii v_\jj = (-1)^{\pairs{\ii}{\jj}} v_{\ii \triangle \jj}$ for $\ii,\jj\in\II^n$.%
\OMIT{Each $v_{i_k}$, starting from last one, jumps over all $v_j$ with $j<i_k$, so we have $|\ii>\jj|$ transpositions. If $i_k$ was already in $\jj$, it cancels out as $v_{i_k}^2 = 1$. The indices remaining are $\ii\triangle\jj$.}\SELF{$=\epsilon_{\ii\jj} v_{\ii\cup\jj} = v_\ii \wedge v_\jj$ if $\ii,\jj$ disjoint}\CITE{cite[p.76]{Rosen2019}}
For $G \in \bigwedge^p X$ and $H \in \bigwedge^q X$, 
$GH$ can have homogeneous components of grades $|q-p|,|q-p|+2,\ldots,p+q-2,p+q$, 
with $\comp{GH}{p+q} = G\wedge H$.
In a \emph{versor} $M=v_1 v_2\cdots v_p$, for $v_1,\ldots,v_p \in X$, all components have the parity of $p$, so $\hat{M} = (-1)^p M$ and $\check{M} = (-1)^{p(n+1)} M$, but $\tilde{M} \neq (-1)^{\frac{p(p-1)}{2}} M$.
\SELF{$\tilde{M}=v_p\cdots v_2 v_1$}
Blades are versors, as they can be factored into orthogonal vectors.
Using an orthonormal basis, for $L,M,N\in \bigwedge X$ one finds
$(MN)\hat{\,} = \hat{M}\hat{N}$, 
$(MN)\check{\,} = \check{M}\check{N}$, 
$(MN)\tilde{\,} = \tilde{N}\tilde{M}$,
\OMIT{Let $p=|\ii|$, $q=|\jj|$, $r=|\ii \cap \jj|$. \\ 
	$(v_\ii v_\jj)\tilde{\,} = (-1)^{\pairs{\ii}{\jj}} \tilde{v}_{\ii \triangle \jj}$ \\
	$= (-1)^{\pairs{\jj}{\ii} + pq -r} \cdot (-1)^{\frac{(p+q-2r)(p+q-2r-1)}{2}} v_{\ii \triangle \jj}$ \\ 
	$= (-1)^{\frac{p(p-1)}{2} + \frac{q(q-1)}{2}} v_\jj v_\ii$ \\
	$= \tilde{v}_\jj \tilde{v}_\ii$} 
$\inner{M,N} = \comp{\tilde{M}N}{0}$,
and 
$\inner{L,MN} = \inner{\tilde{M}L,N} =\inner{L\tilde{N},M}$.%
\OMIT{$\inner{L,MN} = \comp{\tilde{L}MN}{0} = \comp{(\tilde{M}L)\tilde{\,}N}{0} = \inner{\tilde{M}L,N}$}
The mirror principle holds with Clifford products.

\begin{proposition}
	$G \lcontr H = \comp{\tilde{G}H}{q-p}$ and $G\rcontr H = \comp{G\tilde{H}}{p-q}$, for $G \in \bigwedge^p X$ and $H \in \bigwedge^q X$.
\end{proposition}
\begin{proof}
	For $F \in \bigwedge^{q-p} X$, 
	$\inner{F,G\lcontr H} = \inner{G\wedge F,H} = \inner{\comp{GF}{q},H} = \inner{GF,H} = \inner{F,\tilde{G}H} = \inner{F,\comp{\tilde{G}H}{q-p}}$.
	\OMITL{elements of different grades are orthogonal}
	Likewise for $G\rcontr H$.
\end{proof}


\begin{proposition}\label{pr:contr geom}
	Let $v,w_1,\ldots,w_q \in X$ and $M,N \in \bigwedge X$.
	\begin{enumerate}[i)]
		\item $vM = v \lcontr M + v \wedge M$. \label{it:vM}
		\item $v \lcontr M = \frac{vM - \hat{M}v}{2}$ and $v \wedge M = \frac{vM + \hat{M}v}{2}$. \label{it:v contr M Cliff}
		\SELFL{$M \rcontr v = \frac{Mv - v\hat{M}}{2}$	\\ $v \wedge M = \frac{vM + \hat{M}v}{2}$		\\ Hestenes1984 p.8 (1.27)}
		\item $v \lcontr (MN) = (v\lcontr M)N + \hat{M}(v\lcontr N)$. \label{it:v contr MN}
		\item $v \lcontr (MN) = (v \wedge M)N - \hat{M}(v \wedge N)$.\label{it:v contr MN 2}
		\item $v \wedge (MN) = (v\lcontr M)N + \hat{M}(v\wedge N)$. \label{it:v wedge MN}
		\CITE{Hestenes1984 p.12 (1.41)}
		\item $v \wedge (MN) = (v \wedge M)N - \hat{M}(v \lcontr N)$. \label{it:v wedge MN 2}
		\item $v \lcontr (w_1 w_2 \cdots w_q) = \sum_{i=1}^q  (-1)^{i-1} \inner{v,w_i} w_1 w_2 \cdots w_i' \cdots w_q$,
		where $w_i'$ means $w_i$ is absent. \label{it:v contr ws Cliff}%
		\SELF{Equivalent to \Cref{pr:main tools}\ref{it:v lcontr ws}, as linearity lets us assume $v$, $w_i$'s are elements of an orthogonal basis. Some care is needed if there are repeated elements.}\CITEL{\cite[p.\,9]{Hestenes1984} uses $vw_i = 2\inner{v,w_i} - w_i v$ to move $v$ across $M=w_1w_2\cdots w_q$ and find that $vM - \hat{M}v$ is twice this sum}
		\item $(vM) \lcontr N = (v\lcontr M) \lcontr N + M \lcontr (v \lcontr N)$. \label{it:vM contr N}
	\end{enumerate}
\end{proposition}
\begin{proof}
	\emph{(\ref{it:vM})} Assume $v=v_1$ and $M=v_\ii$ for an orthonormal basis $(v_1,\ldots,v_n)$ and $\ii\in \II^n$.
	If $1\in \ii$ then $v_1 \wedge v_\ii=0$ and $v_1 v_\ii = v_{\ii\backslash 1} = v_1 \lcontr v_\ii$.
	If $1 \notin \ii$, $v_1 \lcontr v_\ii=0$ and $v_1 v_\ii = v_1 \wedge v_\ii$.
	\emph{(\ref{it:v contr M Cliff})} $\hat{M}v = \hat{M} \rcontr v + \hat{M} \wedge v = -v \lcontr M + v \wedge M$.
	\emph{(\ref{it:v contr MN})} $v \lcontr (MN) = \frac{vMN - \hat{M}\hat{N}v}{2} = \frac{vM - \hat{M}v}{2} N + \hat{M} \frac{vN - \hat{N}v}{2}$.
	\emph{(\ref{it:v contr MN 2}--\ref{it:v wedge MN 2})} Similar.
	\emph{(\ref{it:v contr ws Cliff})} Follows from \ref{it:v contr MN} by induction.
	\emph{(\ref{it:vM contr N})} Follows from \ref{it:vM}.
\end{proof}

By \ref{it:v contr MN}, $v\lcontr$ is a graded derivation \wrt the Clifford product as well.
With some rearrangements, 
	\OMIT{and relabeling $\tilde{M}\hat{\,} \mapsto M$}
\ref{it:v contr MN} and \ref{it:v wedge MN} are adjoint formulas,
while \ref{it:v contr MN 2} and \ref{it:v wedge MN 2} are self-adjoint.
Since \ref{it:vM contr N} is not as simple as \Cref{pr:main tools}\ref{it:wedge contractor},
there is no easy formula like \Cref{pr:v1...vk contr M} for when the contractor is a versor. 
We find
$(uvw)\lcontr M = \inner{u,v} w \lcontr M - \inner{u,w} v \lcontr M + \inner{v,w} u \lcontr M + w \lcontr \big(v\lcontr (u\lcontr M)\big)$ for $u,v,w\in X$, but with more vectors it becomes increasingly more complex.%
\SELFR{$(v_1 v_2 v_3 v_4) \lcontr M  
	= \sum_{\ii\in\II_2^4}$ \\ $ \epsilon_{\ii\ii'} \bigl( \inner{v_{i_1},v_{i_2}} \inner{v_{i_1'},v_{i_2'}} M$\\
	$+ \inner{v_{i_1},v_{i_2}} v_{i_2'} \lcontr (v_{i_1'} \lcontr M) \bigr)$ \\
	$+ v_4 \lcontr (v_3 \lcontr (v_2 \lcontr ( v_1 \lcontr M)))$}

\begin{corollary}\label{pr:complet orth Leibniz Cliff}
	Let $M\in \bigwedge V$ and $N\in\bigwedge W$.
	\begin{enumerate}[i)]
		\item If $L\in\bigwedge (W^\perp)$ then $L\lcontr(M N) = (L\lcontr M) N$. \label{it:contr L cperp N Cliff}
		\item If $H\in\bigwedge^p (V^\perp)$ then $H\lcontr(M N) =  \hhat{M}{p} (H\lcontr N)$. \label{it:contr H cperp M Cliff}
	\end{enumerate}
\end{corollary}
\begin{proof}
	\emph{(\ref{it:contr L cperp N})} Linearity and \Cref{pr:v1...vk contr M} let us assume $L=v \in W^\perp$,
	\OMIT{trivial if $L$ is scalar}
	in which case it follows from \Cref{pr:contr geom}\ref{it:v contr MN}.
	\OMIT{\ref{pr:v contr M = 0}}
	\emph{(\ref{it:contr H cperp M})} Similar.
\end{proof}

Likewise, $L\wedge (M N) = (L\wedge M) N$ and $H\wedge(M N) =  \hhat{M}{p} (H\wedge N)$.

\begin{corollary}\label{pr: v contr MN = 0}
	\SELF{Uso em \cite{Mandolesi_Contractions2}}
	For $v\in X$ and nonzero $M\in\bigwedge V$ and $N\in\bigwedge W$ in disjoint subspaces $V$ and $W$, $v\lcontr(MN) = 0 \Leftrightarrow v\lcontr M = v\lcontr N = 0$.
\end{corollary}
\begin{proof}
	As in \Cref{pr: v contr M wedge N = 0}, but using \Cref{pr:contr geom}\ref{it:v contr MN}.
\end{proof}

\begin{theorem}\label{pr:generalized Leibniz Cliff}
	For $v_1,\ldots,v_p\in X$ and $M,N\in\bigwedge X$, 
		\SELF{\ref{pr:contr geom}\ref{it:v wedge MN 2} dá $	v_{1\cdots p} \wedge (M N) = \sum_\ii \epsilon_{\ii\ii'} (v_{\ii} \wedge \hhat{M}{|\ii'|}) (\hhat{N}{|\ii'|} \rcontr \tilde{v}_{\ii'})$ (checar)}
		\SELF{Adjoints: \\ $M (v_{1\cdots p} \wedge N) = \sum_{\ii\in\II^p} \epsilon_{\ii\ii'}\, v_{\ii}\wedge \bigl((\hhat{M}{|\ii|} \rcontr \tilde{v}_{\ii'}) N \bigr)$ \\ $	M (v_{1\cdots p} \lcontr N) = \sum_{\ii\in\II^p} \epsilon_{\ii\ii'}\, v_{\ii'} \lcontr \bigl((\hhat{M}{|\ii'|} \rcontr v_{\ii}) N \bigr)$ \\ Proof: for $L \in \bigwedge X$, $\inner{L,M (v_{1\cdots p} \wedge N)} = \inner{v_{1\cdots p} \lcontr (\tilde{M}L),N}	= \sum_{\ii\in\II^p} \epsilon_{\ii\ii'}\, \inner{(v_{\ii'}\lcontr \hhat{\tilde{M}}{|\ii|})(v_\ii\lcontr L),N}	= \sum_{\ii\in\II^p} \epsilon_{\ii\ii'}\, \inner{L,v_\ii \wedge \bigl( (v_{\ii'}\lcontr \hhat{\tilde{M}}{|\ii|})\tilde{\,}\, N\bigr)}$}
	\begin{align}
		v_{1\cdots p}\lcontr(M N) &= \sum_{\ii\in\II^p} \epsilon_{\ii\ii'} (v_{\ii'}\lcontr \hhat{M}{|\ii|})(v_\ii\lcontr N),\text{ and} \label{eq:Leibniz generalizada Cliff}\\
		%
		%
		v_{1\cdots p} \wedge (M N) &= \sum_{\ii\in\II^p} \epsilon_{\ii\ii'} (\tilde{v}_{\ii}\lcontr \hhat{M}{|\ii'|}) (v_{\ii'} \wedge N). \label{eq:v1p wedge M N} 
		%
	\end{align}
\end{theorem}
\begin{proof}
	The proof of \eqref{eq:Leibniz generalizada Cliff} is like that of \eqref{eq:Leibniz generalizada}, but with \Cref{pr:contr geom}\ref{it:v contr MN} and Clifford products instead of $\wedge$.
		\OMIT{except inside the $v_\ii$'s, which remain blades}
%
	\Cref{pr:contr geom}\ref{it:v wedge MN} 
	gives \eqref{eq:v1p wedge M N} for $p=1$ and, assuming it for $v_{1\cdots (p-1)}$, also
	$v_{1\cdots p} \wedge (M N)
	= v_p \wedge \hat{v}_{1\cdots(p-1)} \wedge (MN)
	= v_p \wedge \sum_{\ii\in\II^{p-1}} \epsilon_{\ii\ii'} (\hat{\tilde{v}}_{\ii}\lcontr \hhat{M}{|\ii'|}) (\hat{v}_{\ii'} \wedge N)
	= I + II$,
	where
	\begin{align*}
		I &= \sum_{\ii\in\II^{p-1}} \epsilon_{\ii\ii'} \bigl(v_p \lcontr (\hat{\tilde{v}}_{\ii}\lcontr \hhat{M}{|\ii'|})\bigr) (\hat{v}_{\ii'} \wedge N), \text{ and} \\
		II &= \sum_{\ii\in\II^{p-1}} \epsilon_{\ii\ii'} (\hat{\tilde{v}}_{\ii}\lcontr \hhat{M}{|\ii'|})\hat{\,} \, (v_p \wedge \hat{v}_{\ii'} \wedge N),
	\end{align*}
	which are developed as before, using $\hat{\tilde{v}}_\ii \wedge v_p = v_p \wedge \tilde{v}_\ii = \tilde{v}_{\ii p}$.
\end{proof}

\section{Star duality}\label{sc:Star Operators and Duals}

Hodge-like star operators can be defined by contraction with an orientation element (also called a \emph{unit pseudoscalar} or \emph{volume element}).

\begin{definition}\label{df:star duals}
An \emph{orientation} of $X$ is a unit $\Omega\in\bigwedge^n X$, for $n=\dim X$.
It gives \emph{left and right star operators} $\star:\bigwedge X\rightarrow\bigwedge X$, and \emph{left and right duals} of $M\in\bigwedge X$, respectively, via $\lH{\!M} =  \Omega\rcontr M$ and $\rH{M} = M\lcontr \Omega$.
\end{definition}

When we use $\star$, it is implicit an $\Omega$ was chosen.
If an orthonormal basis $(v_1,\ldots,v_n)$ is fixed, assume $\Omega=v_{1\cdots n}$.
A real space has 2 orientations $\pm\Omega$,
but a complex one has a continuum of them, the unit circle in $\bigwedge^n X$.
Our approach is unorthodox:
complex spaces are usually seen as canonically oriented \cite[p.\,25]{Huybrechts2004},
	\CITE{or non-orientable, as there is no choice cite[p.\,419]{Shaw1983}}
with the complex structure inducing a real orientation $\Omega_\R \in \bigwedge^{2n} X_\R$ in the underlying real space $X_\R$. 

We also write $\star_L$ and $\star_R$ for left and right stars.
Having both simplifies the algebra, and the mirror principle switches them.
	\SELF{In formulas holding for any $\Omega$, so we can relabel $\tilde{\Omega}\mapsto \Omega$}
Note that $\star_R$ uses a left contraction, and vice-versa ($\star$ is at the same side as $\Omega$).
By convention, $\star$ takes precedence over products: $\lH{M}\wedge \rH{N}$ means $(\lH{M})\wedge (\rH{N})$.

In the real case, our $\rH{M}$ and $\lH{M}$ correspond, respectively, to 
the Hodge dual $*M$ and $*^{-1} M$ \cite{Abraham1988,Huybrechts2004},
	\CITE{Griffiths1994 tem erros} 
to $M*$ and $*M$ in \cite{Rosen2019}, 
	\CITE{$*M$, $*'M$ in [p.415]{Shaw1983}. Poincaré isomorphisms [p.354] $\perp\! \alpha$, $\perp'\! M$ are similar, but map forms to multivectors, viceversa}
and, up to signs (see \Cref{sc:geometric algebra}), to the undual $M^{-*}$ and dual $M^*$ of GA \cite{Dorst2007}.
	\SELF{$M^*=M\glcontr I^{-1}$ \\ $M^{-*}=M\glcontr I$}
In the complex case, our stars are conjugate-linear, simpler than the Hodge star of complex analysis%
\footnote{A $\C$-linear (sometimes conjugate-linear, and denoted  
by $\bar{*}$) extension of the real Hodge star of the dual space $X_\R'$ of $X_\R$ to its complexification $X_\R' \otimes_\R \C$, relating $\C$-valued $\R$-linear $p$- and $(2n-p)$-forms \cite[pp.\,156--159]{Wells1980}.}
\cite{Wells1980},
and better suited for complex geometry, relating blades of orthogonal
	\SELF{\wrt Hermitian product}
complex subspaces.

\begin{proposition}\label{pr:Hodge geometric charac}
	For a $p$-blade $B\neq 0$, $\rH{B}$ is the unique $(n-p)$-blade such that $[\rH{B}] = [B]^\perp$, $\|\rH{B}\| = \|B\|$ and $B\wedge \rH{B}$ has the orientation of $\Omega$.
\end{proposition}
\begin{proof}
	Follows from \Cref{pr:characterization 1}.
\end{proof}

Likewise for $\lH{B}$, except that $\lH{B}\wedge B$ has the orientation of $\Omega$.
Note that $[\rH{0}] = [0] \neq [0]^\perp$.

\begin{proposition}\label{pr:star basis}
	$\rH{v_\ii} = \epsilon_{\ii\ii'}  v_{\ii'}$ and $\lH{v_\ii} = \epsilon_{\ii'\ii}  v_{\ii'}$
	for $\ii\in\II^n$, an orthonormal basis $(v_1,\ldots,v_n)$, and orientation $v_{1\cdots n}$.
		\SELF{$\rH{1} = \lH{1 = \Omega}$ and $\rH{\Omega} = \lH{\Omega} = 1$}
\end{proposition}
\begin{proof}
	Follows from \eqref{eq:contr induced basis}.
\end{proof}


\begin{example}
	Let $(v_1,\ldots,v_4)$ be an orthonormal basis of $\C^4$, 
	and $B=3v_1+\im v_3+v_4$.
	Then
	$\rH{B} = 3\rH{v_1\!}-\im\rH{v_3\!}+ \rH{v_4\!} = 3v_{234}-\im v_{124} - v_{123} = v_2 \wedge (v_1-3 v_4) \wedge (v_3 + \im v_4)$,
	and so $[B]^\perp = [v_2]\oplus[v_1-3v_4] \oplus [v_3 + \im v_4]$. 
	Also, $B\wedge \rH{B} = 11\Omega$.
		\SELF{$=\|B\|^2\Omega$.}	
\end{example}

\begin{proposition}\label{pr:Hodge properties}
	Let $M\in\bigwedge X$.
		\SELF{$\rH{\rH{M}}= \check{M}$}
	\begin{enumerate}[i)]
		\item $(\rH{\!M})\check{\,} =  \rH{\!(\check{M})}$. \label{it:Hodge involutions}
		\item $\lH{\!\rH{M}} = M$. \label{it:star inverses}
			\SELF{$\lH{(\rH{M})} = \rH{(\lH{\!M})} = M$}
		\item $\lH{\!M} = \rH{\check{M}}$. \label{it:left right star} 
	\end{enumerate}
\end{proposition}
\begin{proof}
	    \OMITL{\ref{it:star inverses} usa \ref{pr:star basis} }
	\emph{(\ref{it:Hodge involutions})} $(\rH{M})\check{\,} = (M\lcontr\,\Omega)\check{\,} = \check{M}\lcontr \check{\Omega} = \check{M}\lcontr \Omega = \rH{(\check{M})}$.
		\SELF{$M^*=\lambda\rH{M}$ if $*$ is star for $\lambda\,\Omega$, with $\lambda$ a unit scalar}
	\emph{(\ref{it:star inverses})} Follows from \Cref{pr:P B}.
	\emph{(\ref{it:left right star})} Follows from \Cref{pr:contr homog}\ref{it:contr left right M}.
\end{proof}

This shows $\star_L$ and $\star_R$ are inverses,
\SELF{so $\bigwedge^{n-1}$ has only blades}
and lets us (re)position $\star$ as needed.
As $\rH{\rH{M}}= \check{M}$, if $n$ is odd then $\star_L=\star_R$ is an involution of $\bigwedge X$.
If $n$ is even this holds in $\bigwedge^+ X$.
	\SELF{$\check{M} = M$ if $n$ odd or $M$ even, nesse caso precisa $n$ even pra não sair de $\bigwedge^+ X$}
While $*$ and $\check{\ }$ commute, 
$(\rH{\!M})\hat{\,} = (\hat{M}) {}^{\hat{\star}} = (-1)^n\, \rH{\!(\hat{M})}$ 
and 
$(\rH{\!M})\tilde{\,} = \tensor[^{\tilde{\star}}]{}{} (\tilde{M}) = (-1)^{\frac{n(n-1)}{2}}\, \lH{(\tilde{M})}$, 
with stars $\hat{\star}$ and $\tilde{\star}$ for $\hat{\Omega}$ and $\tilde{\Omega}$.
	\SELF{As $\tilde{\,}$ turns a right $\star$-dual into a left $\tilde{\star}$-dual, and vice-versa, relabeling $\tilde{\star} \mapsto \star$ we find that the mirror principle holds switching left and right stars. \\
	Cliff conj:  $(\rH{\!M})\bar{\,} = \tensor[^{\bar{\star}}]{}{} (\bar{M}) = (-1)^{\frac{n(n+1)}{2}}\, \lH{(\bar{M})}$}

\begin{proposition}\label{pr:Hodge wedge contr}
	Let $M,N\in\bigwedge X$.%
		\CITEL{GA duality 		
		\cite[p.\,82]{Dorst2007}, \\
		$(A\wedge B)^* = A\glcontr(B^*)$, \\
		$(A\glcontr B)^* = A\wedge (B^*)$. \\
		{Dorst2007} (p. 82, 79, 594) proves 2nd via \ref{pr:triple subblade}\ref{it:MAB}. Hestenes uses associat of geometr prod}
		\SELF{$\lH{(M\wedge N)} = \, \lH{N}\rcontr M$, \\ 
		$\lH{(M\lcontr N)} = \, \lH{N}\wedge M$,\\
		$\lH{(M\rcontr N)} = \check{N}\wedge\!\lH{M}$,\\ 
		$\rH{(M\lcontr N)} = \rH{N}\!\wedge\check{M}$, \\[3pt] 
		$\lH{M} \lcontr \lH{N} = M\rcontr N$,\\ 
		$\rH{M}\rcontr\rH{N} = M\lcontr N$,\\
		$\rH{M}\lcontr N = \check{M} \rcontr \rH{N}$,\\ 
		$\rH{M}\lcontr\rH{N} = \check{M}\rcontr\check{N}$, \\[3pt]
		$\inner{\lH{M},\,\!\lH{N}} = \inner{N,M}$}
	\begin{enumerate}[i)]
		\item $\rH{(M\wedge N)} = N\lcontr\rH{M}$ and $\rH{(M\rcontr N)} = N\wedge\rH{M}$. \label{it:star wedge contr}
		\item $\rH{M} \rcontr N = M \lcontr \lH{N}$ and $\lH{M} \lcontr N = M \rcontr \rH{N}$. \label{it:star lcontr rcontr}
		\item $\inner{\rH{M},\rH{N}} = \inner{N,M}$. \label{it:star anti-unitary} 
	\end{enumerate}
\end{proposition}
\begin{proof}
		\OMIT{\ref{it:star wedge contr} usa \ref{pr:Hodge properties}\ref{it:star inverses} (ou \ref{pr:triple subblade}\ref{it:MAB} with $B=\Omega$, linearity on $A$)\\
		\ref{it:star lcontr rcontr} usa triple prod, \ref{pr:Hodge properties}\ref{it:star inverses}  \\
		\ref{it:star anti-unitary} usa \ref{it:star lcontr rcontr} }
	\emph{(\ref{it:star wedge contr})} $(M\wedge N)\lcontr\Omega = N\lcontr(M\lcontr\Omega)$, and, with its mirror, $\rH{(M\rcontr N)} = \rH{(\lH{\rH{M}}\rcontr N)} = \lH{\rH{(N\wedge \rH{M})}} = N\wedge \rH{M}$.
	\emph{(\ref{it:star lcontr rcontr})} $(M\lcontr \Omega)\rcontr N = M\lcontr (\Omega\rcontr N)$,
	and $\lH{M} \lcontr N = \lH{M} \lcontr \lH{\!\rH{N}} = \lH{\!\rH{M}} \rcontr \rH{N} = M \rcontr \rH{N}$.
	%
	%
	\emph{(\ref{it:star anti-unitary})} As elements of distinct grades are orthogonal, we can assume $M,N\in\bigwedge^p X$, so that $\inner{N,M} = N \lcontr M = N \lcontr \lH{\!\rH{M}} = \rH{N}\rcontr\rH{M} = \inner{\rH{M},\,\!\rH{N}}$.
\end{proof}


So, $\star$ can turn $\wedge$ into $\lcontr$ or $\rcontr$ and vice-versa, or switch $\lcontr$ and $\rcontr$, but one must heed the sides.
Other formulas are obtained via mirror principle and \Cref{pr:Hodge properties}, like 
$\rH{(M\lcontr N)} = \rH{N}\!\wedge\check{M}$ and
$\rH{M}\rcontr\rH{N} = M\lcontr N$.
The complex case can help identify errors: e.g. $\rH{(M \lcontr N)} \neq N\wedge\rH{M}$, as the first one is linear in $M$, and the other is conjugate-linear.
The reordering in \ref{it:star wedge contr} is again due to the adjoint nature of contractions. 
The corresponding formulas in GA \cite[p.\,82]{Dorst2007} avoid it, and use only the left contraction, but this simplicity comes at the cost of extraneous signs (see \Cref{sc:geometric algebra}).

By \ref{it:star wedge contr}, a subspace can be given by a blade $B$ or its dual $\rH{B}$, as the solution set of $v\wedge B=0$ or $v\lcontr \rH{B} = 0$ (like a plane given by a normal vector in $\R^3$).
	\CITE{Dorst2007 p.\,83}
In \cite{Mandolesi_Contractions2}, we study such equations for  $M\in \bigwedge X$.
By \ref{it:star anti-unitary}, stars are orthogonal operators (anti-unitary, in the complex case), giving isometries (anti-isometries, 
	\CITE{Topics in Math of QM, Hermann 1973 p.74, Google Books. Other authors use term for when metric signature is inverted}
in the complex case) of $\bigwedge^p X$ with $\bigwedge^{n-p} X$.
In the real case, $\star_L$ is the adjoint of $\star_R$
	\OMIT{$\inner{\lH{M},N} = \inner{\rH{N},M} = \inner{M,\rH{N}}$ \\ 
	Or orthog: inverse=adjoint}
(if $n$ is odd, $\star$ is self-adjoint).

\begin{corollary}\label{pr:norm star}
	$\|\rH{M}\| = \|M\|$, for $M\in\bigwedge X$.
\end{corollary}

\begin{corollary}\label{pr:star wedge inner}
	$G\wedge\rH{H} = \inner{H,G}\, \Omega$, for $G,H \in\bigwedge^p X$. 
		\SELFL{$\lH{G}\wedge H = \inner{G,H}\, \Omega$ and $\inner{G,H} = \lH{(\lH{H}\wedge G)}$}
		\OMIT{$G\wedge \rH{H} = \rH{(H\rcontr G)} = \rH{\inner{G,H}} = \inner{H,G}\, \rH{1}$. \ref{pr:Hodge wedge contr}\ref{it:star wedge contr}}
\end{corollary}

Some authors use this to define stars, or to, given a star, determine an inner product by $\inner{G,H} = \rH{(G\wedge\rH{H})}$.
	\SELF{$=\rH{(\lH{\!H}\wedge G)}=\lH{\!(G\wedge\rH{H})}$}
	\OMIT{$\rH{(G\wedge\rH{H})} = \rH{(\inner{H,G}\, \Omega)} = \inner{G,H}\rH{\Omega}$, by \ref{pr:star wedge inner}}

\begin{example}
	In $\R^3$, the cross product is $u \times v = \rH{(u \wedge v)} = v \lcontr \rH{u}$.
	The usual triple product is $\inner{u,v\times w} = u \lcontr \rH{(v\wedge w)} = \rH{(u \wedge v\wedge w)}$, whose modulus is the volume of $u\wedge v\wedge w$.
	We also easily obtain 
	$(u\times v)\times w = w \lcontr \rH{(u\times v)} = w\lcontr (u\wedge v) = \inner{u,w}v-\inner{v,w}u$,
	and $\inner{u\times v,w \times y} = \inner{\rH{(u \wedge v)} , \rH{(w \wedge y)}} = \inner{u \wedge v , w \wedge y} = \inner{u,w} \inner{v,y} - \inner{u,y} \inner{v,w}$.
		\OMIT{\ref{pr:Hodge properties}, \ref{pr:Hodge wedge contr}\ref{it:star wedge contr},\ref{it:star anti-unitary}}
\end{example}

In the real case, stars and Clifford products are related as follows.

\begin{proposition}
	Let $v, v_1,\ldots,v_p \in X$ and $M,N \in \bigwedge X$.
		\SELF{Mirrors: \\
		$\lH{M} = \Omega \tilde{M}$ \\
		$\lH{(MN)} = \lH{N} \tilde{M}$ \\
		$\lH{(v_1 v_2 \cdots v_p)} = \lH{v_p} \, v_{p-1}\cdots v_1$ \\
		$M (\lH{N}) = (\rH{\tilde{M}}) \tilde{N}$ \\
		$\lH{(Mv)} = \lH{(M\rcontr v)} + \lH{v} \rcontr M$}
	\begin{enumerate}[i)]
		\item $\rH{M} = \tilde{M} \Omega$. \label{it:star Cliff}
		\item $\rH{(MN)} = \tilde{N} \rH{M}$. \label{it:star MN}
		\item $\rH{(v_1 v_2 \cdots v_p)} = v_p\cdots v_2 \,\rH{v_1\!\!}$. \label{it:star vvvvv}
		\item $(\rH{M}) N = \tilde{M} (\lH{\tilde{N}})$. \label{it:M star N}
		\item $\rH{(vM)} = \rH{(v\lcontr M)} + M \lcontr \rH{v}$. \label{it:star vM}
			\SELF{$=-\rH{M}\wedge v + M \lcontr \rH{v}$}
	\end{enumerate}
\end{proposition}
\begin{proof}
	\emph{(\ref{it:star Cliff})} If $M=v_\ii$ and $\Omega=v_{1\cdots n}$ for an orthonormal basis $(v_1,\ldots,v_n)$ and $\ii\in\II^n_q$
	then $\rH{v_\ii\!} = v_\ii \lcontr \Omega = \comp{\tilde{v_\ii} \Omega}{n-q} = \tilde{v_\ii} \Omega$,
	as $\tilde{v_\ii} \Omega = \pm v_{\ii \triangle 1\cdots n} = \pm v_{\ii'}$.
	%
	\emph{(\ref{it:star MN})} $\rH{(MN)} = (MN)\tilde{\,} \, \Omega = \tilde{N}\tilde{M}\Omega$.
		\OMIT{\ref{it:star Cliff}}
	\emph{(\ref{it:star vvvvv})} Follows from \ref{it:star MN}.
	\emph{(\ref{it:M star N})} Follows from \ref{it:star Cliff} and its mirror.
	\emph{(\ref{it:star vM})} Follows from Propositions \ref{pr:contr geom}\ref{it:vM} and \ref{pr:Hodge wedge contr}\ref{it:star wedge contr}.
\end{proof}

Stars with respect to oriented subspaces are also useful.

\begin{definition}\label{df:star duals wrt B}
	A unit $q$-blade $B\in\bigwedge X$ gives left and right stars $\star_B:\bigwedge X\rightarrow\bigwedge [B]$, and left and right duals of $M\in\bigwedge X$ \wrt $B$, respectively, by
	$\lHB{M} = B\rcontr M$ and $\rHB{M} = M\lcontr B$.
		\SELF{$\rHB{M} = \rH{M} \rcontr \rH{B}$}
\end{definition}

For a $p$-blade $A \neq 0$, $\rHB{A}$ is a $(q-p)$-blade, and $\rHB{A} \neq 0 \Leftrightarrow [A]\not\pperp[B]$, in which case $[\rHB{A}] = [A]^\perp \cap [B]$.
For $M\in\bigwedge X$,
$\rHB{\lHB{M}} = P_B M$\OMIT{\ref{pr:P B}}\CITE{[p.22]{Hestenes1984}}
and $\rHB{M} = \hhat{\lHB{M}}{(q+1)}$.
In $\bigwedge [B]$, $\star_B$ has properties as $\star$,
which with $\rHB{M} = \rHB{(P_B M)}$ 
	\CITE{[p.\,85]{Dorst2007}}
can be extended to $\bigwedge X$: \eg
$\inner{\rHB{M},\rHB{N}} = \inner{P_B N, P_B M}$.
	\SELF{$\rHB{\rHB{M}} = P_B \hhat{M}{(p+1)}$, \\ $\rHB{M} \rcontr \rHB{N} = P_B M \lcontr P_B N$. \\ NÃO vale $\rHB{(M\rcontr N)} = (P_B N)\wedge\rHB{(P_B M)}$. Ex: $M=N=v$ ortogonal a $B$}

\begin{proposition}\label{pr:starB}
	Let $B$ be a unit blade, and $M,N\in\bigwedge X$.
	\begin{enumerate}[i)]
		\item $\rH{M} = \rHB{M}\wedge \rH{B}$, if $M \in \bigwedge [B]$. \label{it:star_B 1}
		\item $\rHB{(M \wedge N)} = N \lcontr \rHB{M}$. \label{it:star_B wedge}
		\item $\rHB{(A\rcontr M)} = (P_A M)\wedge \rHB{A}$, if $A$ is a subblade of $B$. \label{it:star_B contr 2}
		\item $\rHB{(M\rcontr N)} = N\wedge\rHB{M}$, if $N \in \bigwedge [B]$. \label{it:star_B contr}
	\end{enumerate}
\end{proposition}
\begin{proof}
	\OMIT{(\ref{it:star_B 1}) \ref{pr:Hodge geometric charac}, \ref{pr:star wedge inner}, \ref{pr:complet orth Leibniz}\ref{it:contr L cperp N} \\ 
		(\ref{it:star_B wedge}) \ref{pr:main tools}\ref{it:wedge contractor} \\
		(\ref{it:star_B contr 2}) \ref{pr:triple subblade}\ref{it:AMB}  \\
		(\ref{it:star_B contr}) \ref{pr:triple subblade}\ref{it:MAB} } 
	\emph{(\ref{it:star_B 1})} $B\wedge\rH{B}= \Omega$, so $\rH{M} = M\lcontr(B\wedge\rH{B}) = (M\lcontr B)\wedge\rH{B}$, as $M \in \bigwedge ([\rH{B}]^\perp)$.
	%
	%
	\emph{(\ref{it:star_B wedge}--\ref{it:star_B contr})} Follow from Propositions \ref{pr:main tools}\ref{it:wedge contractor} and \ref{pr:triple subblade}.
\end{proof}

\subsection{Regressive product}\label{sc:Regressive Product}

Stars induce a regressive product dual to the exterior one.
These two are the basic products of Grassmann-Cayley algebra \cite{White2017}, in which, however, their symbols are usually swapped.
	\CITE{Dorst2007 p.136 diz que é álgebra não métrica. Significa que $\vee$ acaba não dependendo do inner prod?}

\begin{definition}\label{df:regressive}
	The \emph{regressive product} $M\vee N$ of $M,N\in\bigwedge X$ is given by $\rH{(M\vee N)} = \rH{M} \wedge \rH{N}$.
		\SELF{Grassmann combined $\wedge$ e $\vee$, dá um ou outro se $p+q>n$ ou $<n$:  ``Grassmann, Geometry and Mechanics'', Browne p.9}
\end{definition}

It is bilinear, associative,
and satisfies $G\vee H = (-1)^{(n-p)(n-q)} H\vee G \in \bigwedge^{p+q-n} X$ 
for $G\in\bigwedge^p X$, $H\in\bigwedge^q X$ and $n =\dim X$.
Also, $G\vee H=0$ if $p+q<n$.

\begin{proposition}\label{pr:regressive}
	Let $M,N\in\bigwedge X$.
		\SELF{$\lH{(M\vee N)} = \lH{M} \wedge \lH{N}$ \\ 
			$\lH{(M\wedge N)} = \lH{M} \vee \lH{N}$}
	\begin{enumerate}[i)]
		\item $(M\vee N)\check{\,} = \check{M} \vee \check{N}$. \label{it:regr check}
		\item $\rH{(M\wedge N)} = \rH{M} \vee \rH{N}$. \label{it:dual wedge}
		\item $M\vee N =  N \rcontr \rH{M}$. \label{it:regr contr}
			\SELF{$= \lH{N} \lcontr M$, mas $N \rcontr \rH{M}$ é melhor pra \ref{pr:characterization regressive}}
	\end{enumerate}
\end{proposition}
\begin{proof}
	\emph{(\ref{it:regr check})} Follows via \Cref{pr:Hodge properties}\ref{it:Hodge involutions}.
	\emph{(\ref{it:dual wedge})} $\rH{(M\wedge N)} = \rH{(\lH{\rH{M}} \wedge \lH{\rH{N}})} = \rH{\rH{(\lH{M} \vee \lH{N})}} = (\lH{M} \vee \lH{N})\check{\,} = \rH{M} \vee \rH{N}$.
	\emph{(\ref{it:regr contr})} 
	$M\vee N = \rH{(\lH{M} \wedge \lH{N})} = \lH{N} \lcontr \rH{\lH{M}} = \lH{N} \lcontr M =  N \rcontr \rH{M}$.
		\OMIT{\ref{pr:Hodge wedge contr}, \ref{pr:Hodge properties}\ref{it:star inverses} \\ \ref{it:dual wedge}, \ref{pr:Hodge wedge contr}\ref{it:star wedge contr}}
\end{proof}

Also, $(M\vee N)\hat{\,} = \hat{M} \hat{\vee} \hat{N} = (-1)^n \hat{M} \vee \hat{N}$
and $(M\vee N)\tilde{\,} = \tilde{N} \tilde{\vee} \tilde{M} = (-1)^{\frac{n(n-1)}{2}} \tilde{N} \vee \tilde{M}$, where $\hat{\vee}$ and $\tilde{\vee}$ are regressive products \wrt $\hat{\Omega}$ and $\tilde{\Omega}$.
Relabeling $\tilde{\vee}$ as $\vee$, the mirror principle holds.

\begin{proposition}\label{pr:regr induced basis}
	$v_{\ii} \vee v_{\jj} = \delta_{\ii\cup\jj = 1\cdots n} \, \epsilon_{\jj' \ii'} \, v_{\ii\cap\jj}$ for $\ii,\jj\in\II^n$, orthonormal basis $(v_1,\ldots,v_n)$, and orientation $v_{1\cdots n}$. 
\end{proposition}
\begin{proof} 
	$v_{\ii} \vee v_{\jj} = v_\jj \rcontr \rH{v_\ii\!} 
	= \epsilon_{\ii\ii'} \, v_{\jj} \rcontr v_{\ii'}$
	is $0$ 
	unless $\ii\cup\jj = 1\cdots n$ (so $\ii' \subset \jj$), in which case $\jj = \ord{(\ii \cap \jj)\ii'}$ and $v_{\ii} \vee v_{\jj} = \epsilon_{\jj'\ii'} \, v_{\ii \cap \jj}$, 
	as \Cref{pr:epsilon} gives
	$\epsilon_{\ii\ii'} \, \epsilon_{(\ii \cap \jj)\ii'}
	= \epsilon_{\ord{(\ii \cap \jj)\jj'}\, \ii'} \, \epsilon_{(\ii \cap \jj)\ii'}
	= \epsilon_{(\ii \cap \jj)\jj' \ii'} \, \epsilon_{(\ii \cap \jj) \jj'} \,\epsilon_{(\ii \cap \jj) \ii'}
	= \epsilon_{\jj'\ii'}$.
		\OMIT{\ref{pr:regressive}\ref{it:regr contr}, \ref{pr:star basis}, \eqref{eq:contr induced basis}, \ref{pr:epsilon}\ref{it:epsilon ord ra},\ref{it:epsilon r1 rm}}
\end{proof}

The following geometric characterization  dualizes the fact that, for nonzero blades, $A\wedge B\neq 0 \Leftrightarrow [A]\cap [B] = \{0\}$, in which case $[A\wedge B] = [A]\oplus[B]$,
and $\|A\wedge B\| = \|A\|\|B\|\cos\Theta_{[A],[B]^\perp}$ 
	\OMIT{$= \|(P_{[B]^\perp} A)\wedge B\| = \|P_{[B]^\perp} A\| \| B\|$}
(see \cite{Mandolesi_Products}).

\begin{theorem}\label{pr:characterization regressive}
	For nonzero blades $A\in\bigwedge^p X$ and $B\in\bigwedge^q X$, $A\vee B = 0 \Leftrightarrow [A]+[B] \neq X$.
	If $A\vee B \neq 0$, it is the only ($p+q-n$)-blade with:
	\begin{enumerate}[i)]
		\item $[A\vee B] = [A]\cap [B]$. \label{it:space regr}
		\item $\|A\vee B\| = \|A\|\|B\|\cos\Theta_{[A]^\perp,[B]}$. \label{it:norm regr} 
		\item $A \wedge \big((A \vee B)\lcontr B \big)$ has the orientation of $\Omega$.\label{it:orient regr}
			\SELF{or $\big(A \rcontr (A \vee B)\big) \wedge B$}
	\end{enumerate}
\end{theorem}
\begin{proof}\OMIT{\ref{pr:Hodge geometric charac}}
	$\rH{(A \vee B)} = \rH{A}\wedge \rH{B} = 0 \Leftrightarrow \{0\} \neq [\rH{A}] \cap [\rH{B}] = ([A]+[B])^\perp$.
	If $A\vee B = B \rcontr \rH{A} \neq 0$ then $[A\vee B] = [\rH{A}]^\perp \cap [B]$,
		\OMIT{\ref{pr:regressive}\ref{it:regr contr}, \ref{pr:characterization 1}\ref{it:space contr}}
	$\inner{\Omega,A \wedge ((A \vee B)\lcontr B)} = \inner{\rH{A}, B \rcontr \rH{A} \lcontr B} = \inner{\rH{A}, \|B\|^2 P_B(\rH{A})} = \|B\|^2 \|P_B(\rH{A})\|^2$,
	and $P_B(\rH{A}) \neq 0$.
	Also, 	$\|A\vee B\| = \|B \rcontr \rH{A}\| =  \|\rH{A}\|\|B\|\cos\Theta_{[\rH{A}],[B]}$.
\end{proof}

So, $A\vee B$ describes `necessary' intersections, that occur once the subspaces fill up $X$.
As $[A]+[B]\neq X \Leftrightarrow [A]^\perp \pperp [B] \Leftrightarrow \Theta_{[A]^\perp,[B]} = \frac\pi2$,
\ref{it:norm regr} holds for all blades.
In general, $\Theta_{V^\perp,W} \neq \Theta_{V,W^\perp} \neq \frac\pi2 - \Theta_{V,W}$, but $\Theta_{V^\perp,W} = \Theta_{W^\perp,V}$ and $\Theta_{V,W^\perp} = \Theta_{W,V^\perp}$ \cite{Mandolesi_Grassmann}.
If $C=A\vee B \neq 0$ then
$A = A' \wedge C$ and $B = C \wedge B'$ for $A'= \frac{A \rcontr C}{\|C\|^2}$ and $B' = \frac{C\lcontr B}{\|C\|^2}$,
and \ref{it:orient regr} means $A\wedge B' = A' \wedge C \wedge B' = A' \wedge B$ has the orientation of $\Omega$ (Fig.\,\ref{fig:regressive}).

\begin{figure}
	\centering
	\includegraphics[width=0.4\linewidth]{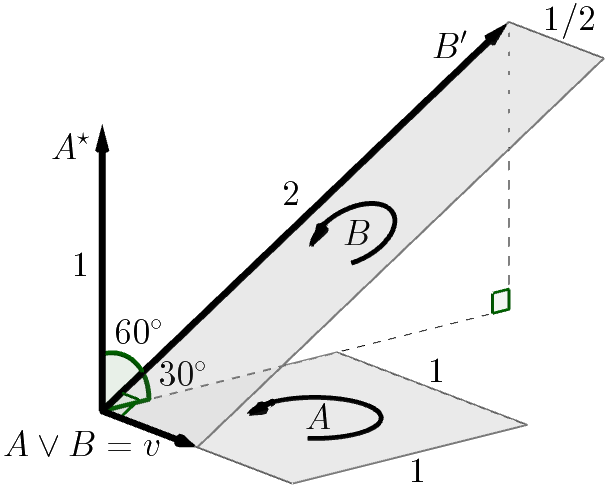}
	\caption{For unit blades $A,B \in \bigwedge^2 \R^3$ with $\Theta_{[A]^\perp,[B]} = 60^\circ$, $A\vee B$ is a vector $v$ of norm $\frac12$ in $[A] \cap [B]$, oriented so if $B=v \wedge B'$ then $A\wedge B'$ has the orientation of $\R^3$.}
	\label{fig:regressive}
\end{figure}

\begin{example}\label{ex:regr 1}
	If $A=v_{13}$ and 
	$B=(v_1+v_2)\wedge(v_3+\sqrt{3}v_4)$
		\OMIT{$(v_1,v_2)$ and $(w_1,w_2)$, with $w_1=\frac{v_1+v_3}{\sqrt{2}}$ and $w_2=\frac{v_2+\sqrt{3}v_4}{2}$, are principal bases of $V$ and $W$, with $\theta_1=45^\circ$ and $\theta_2=60^\circ$.}
	for an orthonormal basis $(v_1,\ldots,v_4)$ of $\R^4$	then $A\vee B = \sqrt{3} v_{13} \vee v_{24} = \sqrt{3} \, \epsilon_{1324} \, v_\emptyset = -\sqrt{3}$,
	so $V\cap W = \{0\}$ for $V=[A]$ and $W=[B]$. 
	Also, $A \wedge ((A \vee B)\lcontr B) 
	= 3 v_{1234}$,
	and areas of $V^\perp$ orthogonally projected on $W$ contract by $\cos \Theta_{V^\perp,W} = \frac{\sqrt{3}}{1\cdot 2\sqrt{2}} = \frac{\sqrt{6}}{4}$.
	Note that
	$\cos \Theta_{V,W^\perp} = \frac{\|A\wedge B\|}{\|A\|\|B\|} =  \frac{\sqrt{6}}{4}$
	($\Theta_{V,W^\perp} = \Theta_{W,V^\perp} = \Theta_{V^\perp,W}$ as $\dim V^\perp =\dim W$),
	but $\cos \Theta_{V,W} = \frac{|\inner{A,B}|}{\|A\|\|B\|} = \frac{\sqrt{2}}{4}$, and so $\Theta_{V,W} + \Theta_{V^\perp,W} \neq \frac\pi2$.
%
\end{example}

\begin{example}\label{ex:regr 3}
	If $A=v_1\wedge(v_2+v_4)\wedge(\im v_3+ v_4) = \im v_{123} + v_{124} - \im v_{134}$ and $B=v_{123}$
	for an orthonormal basis $(v_1,\ldots,v_4)$ of $\C^4$ 
	then $A\vee B = \epsilon_{43} \, v_{12}-\im \epsilon_{42} \, v_{13} = - v_{12} + \im v_{13}$,
	so $V\cap W = [v_1]\oplus[\im v_3- v_2]$ 
	for $V=[A]$ and $W=[B]$.
	Also, 
	$A = A' \wedge C$ for $C = A\vee B$ and
	$A' = \frac{A\rcontr C}{\|C\|^2} = - \frac{v_2 +2v_4+\im v_3}{2}$,
	$B=C\wedge B'$ for $B' = \frac{C\lcontr B}{\|C\|^2} = \frac{-v_3 + \im v_2}{2}$, 
	and $A\wedge B' = A'\wedge B = v_{1234}$.
	As $\cos \Theta_{V^\perp,W} = \frac{\sqrt{2}}{\sqrt{3}}$, areas of $V^\perp$ orthogonally projected on $W$ contract by $\frac{2}{3}$.
	As $\cos \Theta_{V,W} = \frac{1}{\sqrt{3}}$ and $\Theta_{V,W^\perp}=90^\circ$,
		\OMIT{$\inner{A,B} = 1$ and $A\wedge B = 0$}
	6-dimensional volumes of $V$ contract by $\frac{1}{3}$ if orthogonally projected on $W$, and vanish on $W^\perp$.%
		\SELF{As $\dim V^\perp = \dim [A'] = \dim [B'] =1$,	$\Theta_{V,W} + \Theta_{V^\perp,W} = \Theta_{[A'],[B']} + \Theta_{V^\perp, [B']} = \frac\pi2$.}
		\OMIT{$\Theta_{V^\perp,W} = \Theta_{V^\perp, [B']} = \Theta_{[B'],V^\perp} = \frac\pi2 - \Theta_{[B'],V}$, and $\Theta_{[B'],V} = \Theta_{W,V} = \Theta_{V,W}$ \\
		\ \\
		$V$ and $W$ are 3-dimensional complex subspaces of $\C^4$,
		$\dim (V\cap W) = 2$ and $\Theta_{V^\perp,W} = \cos^{-1} \sqrt{\frac23} \cong 35^\circ$,
		so $\theta_1=\theta_2=0$, $\theta_3\cong 55^\circ$}
\end{example}

\subsubsection{\join\ and \meet}\label{sc:join and meet}

The usefulness of $\vee$ for finding intersections is limited by the condition $[A]+[B] = X$, which is too restrictive for low grade blades in large spaces.
A workaround is to reduce the space to $[A]+[B]$, in which case another notation is used \cite{Dorst2007} 
	\CITE{Bouma2002}
(but note that some authors use the terms join and meet for $\wedge$ and $\vee$).

\begin{definition}
	The \join\ $A\cup B$ and \meet\ $A\cap B$ of blades $A$ and $B$ are defined, up to scalar multiples, as nonzero blades $J$ and $M$, respectively, such that $[J] = [A]+[B]$ and $[M] = [A]\cap [B]$.
\end{definition}

These operations are nonlinear and restricted to blades.
If a \join\ $J$ is known, a \meet\ can be obtained as in the regressive product, but with $\star_J$.
And given a \meet\ we can obtain a \join.

\begin{proposition}
	Let $A,B\in\bigwedge X$ be nonzero blades.
		\SELF{Start with a $J$, find \meet, find \join, and normalize, returns the same $J$.}
	\begin{enumerate}[i)]
		\item Given a unit \join\ $J$, a \meet\ is 
		$A\cap B = B \rcontr \rH[\star_J]{A}$.\label{it:meet from join}
			\SELF{$= \lH[\star_J]{B} \lcontr A$}
		\item Given a \meet\ $M$, a \join\ is $A\cup B = A \wedge (M \lcontr B)$.\label{it:join from meet}
			\SELF{$= (M\lcontr A)\wedge B$}
	\end{enumerate}
\end{proposition}
\begin{proof}
	\emph{(\ref{it:meet from join})} 
	$[A]\cup[B]=[J]$ implies
		\OMIT{$[B] \not\pperp [J]$}
	$[\rH[\star_J]{A}] = [A]^\perp \cap [J] \not\pperp [B]$, so
	$[B \rcontr \rH[\star_J]{A}] = [\rH[\star_J]{A}]^\perp \cap [B] = ([A] + [J]^\perp) \cap [B] = [A] \cap [B]$.
		\OMIT{or \ref{pr:regressive}\ref{it:regr contr}, with $X=[J]$, and \ref{pr:characterization regressive}\ref{it:space regr}}
	\emph{(\ref{it:join from meet})} $B=M\wedge B'$
	 with $[B'] \perp [M]$,
	so $M\lcontr B = \|M\|^2 B'$ and
	$[A \wedge (M \lcontr B)] = [A \wedge B'] = [A]+[B]$.
\end{proof}

The \meet\ $A \cap B$ in \ref{it:meet from join} has properties like $A \vee B$, but considering the space as $[J]$:
e.g., $\|A \cap B\| = \|A\|\|B\| \cos \Theta_{[A]^\perp \cap [J],[B]}$.
Its bilinear formula remains valid if $A$ or $B$ changes but $[A]+[B]$ does not.

\begin{example}
	If $A=(v_1+v_4)\wedge(v_2+2v_4)$
		\OMIT{$= v_{12}+2v_{14}-v_{24}$}
	and $B=v_{12}$ for an orthonormal basis $(v_1,\ldots,v_4)$ of $\R^4$ then $A\vee B=0$ does not give $[A]\cap[B]$.
	As $[A]$ and $[B]$ are distinct planes in $[v_{124}]$, a unit \join\ is $J=v_{124}$, 
	a \meet\ is 
	$M = B \rcontr (A \lcontr J) = v_{12} \rcontr (v_4 - 2 v_2 - v_1) = -2v_1 + v_2$,
	and $[A]\cap[B] = [v_2-2v_1]$.
	Lengths in $[A]^\perp \cap [J] = [A \lcontr J] = [v_4 - 2v_2 - v_1]$ contract by $\cos \Theta_{[A]^\perp \cap [J],[B]} =  \frac{\sqrt{5}}{\sqrt{6}}$ if orthogonally projected on $[B]$.
\end{example}

\subsection{Outermorphisms and duality}

Let $T:X\rightarrow Y$ be a linear map into a Euclidean or Hermitian space $Y$ (same as $X$)
with orientation $\Omega_Y$ and star $*$ (beware: $\star$ is in $X$, $*$ in $Y$). 
We have $T\Omega \neq 0 \Leftrightarrow T$ is injective. 
If $T$ is invertible 
\SELF{vector space isomorphism}
then $T\Omega = \Delta_T \cdot \Omega_Y$ for a scalar $\Delta_T \neq 0$.
\SELF{depende também dos $\Omega$'s}
If $Y=X$ then $T\Omega = (\det T) \cdot \Omega$.%
\SELF{$T^{-1} M = \frac{1}{\det T} \lH{\big(\adj{T}(\rH{M})\big)}$
	as usual \cite[p.\,62]{Rosen2019}, \cite[p.113]{Dorst2007}}

\begin{proposition}\label{pr:T star geral}
	$T\big(\rH{(\adj{T} N)}\big) = N \lcontr T\Omega$, for $N\in\bigwedge Y$.
	If $T$ is injective then
	$T\big(\rH{(\adj{T} N)}\big) = \|T\Omega\|\, \rH[*_B]{N}$ with $B=\frac{T\Omega}{\|T\Omega\|}$, otherwise $T\big(\rH{(\adj{T} N)}\big) = 0$.
\end{proposition}
\begin{proof}
	Follows from \Cref{pr:T contr adj}.
\end{proof}

\begin{corollary}
	If $T$ is an isometry then $T(\rH{M}) = \rH[*_{T\Omega}]{(TM)}$, for $M \in \bigwedge X$.
	If it is also invertible then $T(\rH{M}) = \Delta_T \cdot \rH[*]{(TM)}$ and $|\Delta_T| = 1$.
\end{corollary}
\begin{proof}
	$M = \adj{T} N$ for $N = TM$, and $\|T\Omega\|=1$.
	\OMIT{\ref{pr:T star geral}}
\end{proof}

\begin{corollary}\label{pr:T star}
	If $T$ is invertible, the following diagram is commutative. 
	\SELF{Applied to blades, this gives the usual result that if $T$ is invertible and $T(V)=W$ then $\adj{T}(W^\perp) = V^\perp$.}
	Stars can be left or right, but must have equal (\resp opposite) sides for equal (\resp opposite) arrows.
	\SELFL{$\adj{T} \big(\rH[*]{(TM)}\big) = \vf{T} \rH{M}$ \\ $T\big(\rH{(\adj{T}N)}\big) = \vf{T} \rH[*]{N}$ \\ $\adj{T} \big(\rH[*]{N}\big) = \vf{T}\,\rH{(T^{-1}N)}$ \\ $\rH{(\adj{T}N)} = \vf{T}\,T^{-1}(\rH[*]{N})$ \\ If $T$ is orthogonal, $T(\rH{M}) = \rH[*]{(TM)}$}
	\begin{equation*}
		\begin{tikzcd}
			\bigwedge X \arrow[d,shift left=1,"\frac{T}{\Delta_T}"] \arrow[r,"\star"] & \bigwedge X \arrow[l] \arrow[d,shift left=1,"(\adj{T})^{-1}"] \\
			\bigwedge Y \arrow[u,shift left=1,"\Delta_T \cdot T^{-1}"] \arrow[r] & \bigwedge Y \arrow[u,shift left=1, "\adj{T}"] \arrow[l,"*"]
		\end{tikzcd}
	\end{equation*}
\end{corollary}
\begin{proof}
	$T\big(\rH{(\adj{T} N)}\big) = \Delta_T \cdot \rH[*]{N}$ for $N\in\bigwedge Y$, as above.
		\OMIT{\ref{pr:T contr adj}}
	Other relations follow from it, like $T(\rH{M}) = \Delta_T \cdot \rH[*]{\big((\adj{T})^{-1}M\big)}$ for $M = \adj{T} N\in \bigwedge X$.
	\OMIT{\ref{pr:T star geral}, \ref{pr:Hodge properties}\ref{it:star inverses}}
\end{proof}

\begin{corollary}\label{pr:inverse adjoint}
	$T^{-1} N = \frac{1}{\Delta_T}  \lH{\big(\adj{T}(\rH[*]{N})\big)}$ for $N\in\bigwedge Y$, if $T$ is invertible.%
		\OMIT{\ref{pr:T star}}\CITE{Pavan p.153: operador $T^\sharp M = \lH{T(\rH{M})}$, tal que $T^{-1} = \frac{1}{\vf{T}} (\adj{T})^\sharp$ \\ p.162: cofactor formula}	
\end{corollary}

\begin{example}
	If $T=\left(\begin{smallmatrix}
		0 & 1 & -1 \\
		2 & \im & 0 \\
		1 & 0 & -\im
	\end{smallmatrix}\right)$,
	$N=w_1+w_{13}$, $\Omega = v_{123}$ and $\Omega_Y = w_{123}$
	in orthonormal bases $(v_1,v_2,v_3)$ of $X$ and $(w_1,w_2,w_3)$ of $Y$
	then $\Delta_T=3\im$,
	$\rH[*]{N} = w_{23}-w_2$ 
	and
	$\adj{T}(\rH[*]{N}) = (2v_1-\im v_2)\wedge(v_1+\im v_3)-(2v_1-\im v_2)$,
	\OMIT{$\adj{T}=\left(\begin{smallmatrix}
			0 & 2 & 1 \\
			1 & 1 & 0 \\
			-1 & 0 & 2
		\end{smallmatrix}\right)$}
	so 
	$T^{-1}N = \frac{1}{3\im} \lH{(2\im v_{13} + \im v_{12} + v_{23} - 2v_1 + \im v_2)} = \frac{2 v_2 - v_3 - \im v_1 + 2\im v_{23} + v_{13}}{3}$. 
\end{example}

\begin{corollary}\label{pr:outer Y=X}
	Let $T:X \rightarrow X$ be linear, and $M\in\bigwedge X$.
	\begin{enumerate}[i)]
		\item $T\big(\rH{(\adj{T}M)}\big) = (\det T)\cdot \rH{M}$.
			\CITE{\cite[p.\,61, 2.7.2]{Rosen2019} similar result using dual spaces.}
			\SELF{For $M\neq 0$, $\adj{T}\big(\rH{(TM)}\big) \neq 0 \Leftrightarrow T$ is invertible.}
		\item $T(\rH{M}) = (\det T)\cdot\rH{\big((\adj{T})^{-1}M\big)}$ if $T \in GL(X)$. \label{it:T adj star}
		\item $T(\rH{M}) = (\det T)\cdot\rH{(TM)}$ if $T \in U(X)$.\label{it:O U}
		\item $T(\rH{M}) = \rH{(TM)}$ if $T \in SU(X)$.
	\end{enumerate}
\end{corollary}

In the real case, the general and special unitary groups $U(X)$ and $SU(X)$ become the orthogonal ones, $O(X)$ and $SO(X)$.

\begin{proposition}
	If $T:X \rightarrow Y$ is invertible and $M,N\in \bigwedge X$ then $T(M \vee N) = \frac{1}{\Delta_T} TM \vee TN$,
	where the second $\vee$ is \wrt $*$ (in $Y$).
\end{proposition}
\begin{proof}
		\OMIT{\ref{pr:T star}}
	$T(M \vee N) 
	= T\bigl(\lH{(\rH{M}\wedge\rH{N})}\bigr) 
	= \Delta_T \cdot \lH[*]{\bigl((\adj{T})^{-1}(\rH{M}\wedge\rH{N})\bigr)}
	= \Delta_T \cdot \lH[*]{\bigl((\adj{T})^{-1}(\rH{M})\wedge (\adj{T})^{-1}(\rH{N})\bigr)}
	= \Delta_T \cdot \lH[*]{\bigl(\rH[*]{(\frac{TM}{\Delta_T})}\wedge \rH[*]{(\frac{TN}{\Delta_T})}\bigr)} 
	= \frac{TM \vee TN}{\Delta_T}$.
\end{proof}

\begin{corollary}
	$T(M \vee N) = \frac{1}{\det T} TM \vee TN$, for $M,N\in\bigwedge X$ and $T\in GL(X)$.
\end{corollary}

\section*{Acknowledgments}

The author would like to thank Dr. Leo Dorst of the University of Amsterdam for his comments and suggestions.

%

\appendix

\section{The contractions zoo}\label{sc:Appendix}

The literature has many contractions (also called interior products, inner multiplications, inner derivatives, insertion operators, etc.), all playing similar roles, but with subtle differences which can be confusing.
Here we explain their differences and the reason for such diversity.

\subsection{Different contractions}

In this section, $X$ does not need an inner product at first, $X'$ is its dual space, and we use $A, B, C$ for any multivectors.

The simplest contraction is the pairing of $\varphi\in X'$ and $v\in X$, giving a scalar $\inner{\varphi,v} = \varphi(v)$.
	\SELF{or as $v(\varphi)$ (identifying $(X')'\cong X$)}
For general tensors \cite{Bourbaki1998}, contractions trace out selected pairs $(i,j_i)$ of covariant and contravariant indices, giving a product of pairings $\varphi_i(v_{j_i})$ and a lower order tensor.
For multivectors and forms (multi-covectors), alternativity requires adding the results of all correspondences of indices of one element with those of the other, choosing an initial one as positive, and changing sign for each index transposition.

Let $A=v_1\wedge\cdots\wedge v_p \in\bigwedge^p X$ and $\alpha = \varphi_1\wedge\cdots\wedge \varphi_q \in\bigwedge^q (X')$. 
If $p=q$, their contraction is a pairing $\inner{\alpha,A} = \alpha(v_1,\ldots,v_p) = \det\left(\varphi_i(v_j)\right)$,
	\SELF{also written $\inner{A,\alpha}$, $\alpha(A)$, $A(\alpha)$}
with $\varphi_1(v_1)\cdots \varphi_p(v_p)$ chosen as positive.
Setting $\inner{\alpha, A} = 0$ if $p \neq q$, and extending linearly, we have a nondegenerate pairing of $\bigwedge X$ and $\bigwedge (X')$, and an isomorphism $(\bigwedge X)' \cong \bigwedge (X')$.
Contractions differ from this pairing if $p\neq q$, giving, instead of a scalar, a lower grade multivector or form built with the $\varphi_i$'s or $v_j$'s left out of each index correspondence. 
Two choices for an initial positive correspondence give left or right contractions\footnote{This is convention I of \Cref{sc:Other Conventions}; II switches left and right contractions; and III matches last covectors with first vectors, or vice-versa.}: first covectors with first vectors, or last covectors with last vectors.

For $p=1$, contractions of $v\in X$ on $\alpha$ are $(q-1)$-forms.
	\CITE{Analysis, Manifolds and Physics, p.207, ou Abraham Marsden p.429, also called insertion operator or inner derivation}
The left one, 
$v \lcontr \alpha = \sum_{i=1}^q (-1)^{i-1} \varphi_1\wedge\cdots\wedge \varphi_i(v)\wedge\cdots\wedge \varphi_q$,
matches $v$ positively with $\varphi_1$,
and the right one,
$\alpha \rcontr v = \sum_{i=1}^q (-1)^{q-i}  \varphi_1\wedge\cdots\wedge \varphi_i(v) \wedge\cdots\wedge \varphi_q$,
with $\varphi_q$.
They are partial evaluations: for $u_1,\ldots,u_{q-1}\in X$, $v$ is inserted in the first entry of $\alpha$ in
$(v\lcontr \alpha)(u_1,\ldots,u_{q-1}) = \alpha(v,u_1,\ldots,u_{q-1})$,
or in the last one in
$(\alpha\rcontr v)(u_1,\ldots,u_{q-1}) = \alpha(u_1,\ldots,u_{q-1},v)$.
Equivalently, $\inner{v\lcontr\alpha,B} =\inner{\alpha,v\wedge B}$ and $\inner{\alpha\rcontr v,B} = \inner{\alpha,B\wedge v}$ for $B\in\bigwedge^{q-1}X$.
Likewise, for $q=1$, contractions of $\varphi\in X'$ on $A$ are $(p-1)$-vectors given by $\inner{\beta,\varphi\lcontr A} =  \inner{\varphi\wedge\beta,A}$ and $\inner{\beta,A\rcontr \varphi} = \inner{\beta\wedge\varphi,A}$ for $\beta\in\bigwedge^{p-1}X'$, 
so $\varphi$ is applied positively on $v_1$ for $\varphi\lcontr A$, or on $v_p$ for $ A\rcontr \varphi$.

Generalizing, 
we have four contractions,
$A\lcontr\alpha,\, \alpha\rcontr  A \in\bigwedge^{q-p} X'$ and $\alpha\lcontr  A,\, A\rcontr \alpha\in\bigwedge^{p-q} X$, given, for $B\in\bigwedge^{q-p}X$ and $\beta\in\bigwedge^{p-q}X'$, by
	\begin{align*}
		\inner{A\lcontr\alpha,B} &= \inner{\alpha, A\wedge B}, & \inner{\beta,\alpha\lcontr A} &= \inner{\alpha\wedge\beta,A}, \\ 
		\inner{\alpha\rcontr A,B} &= \inner{\alpha,B\wedge A}, & \inner{\beta,A\rcontr \alpha}&= \inner{\beta\wedge \alpha,A}.
	\end{align*}
They extend linearly for all $A\in\bigwedge X$ and $\alpha\in\bigwedge X'$.
Left (\resp right) contractions match the \emph{contractor} (the  element switching sides) positively with the first (\resp last) components of the \emph{contractee} (the other element).
The result is of the same kind (multivector or form) as the contractee, vanishing if the contractor has larger grade.

\begin{example}
	For $ A\in\bigwedge^3 X$ and $\alpha=\varphi_1\wedge \varphi_2\wedge \varphi_3\wedge \varphi_4\in\bigwedge^4 X'$, 
	we have $\alpha\lcontr A =  A\rcontr\alpha=0$, and
	\begin{align*}
		A\lcontr \alpha = & +\inner{\varphi_1\wedge \varphi_2\wedge \varphi_3, A}\, \varphi_4 - \inner{\varphi_1\wedge \varphi_2\wedge \varphi_4, A}\, \varphi_3 \\
		& +\inner{\varphi_1\wedge \varphi_3\wedge \varphi_4, A}\, \varphi_2 - \inner{\varphi_2\wedge \varphi_3\wedge \varphi_4, A}\, \varphi_1, \\
		\alpha\rcontr  A = & +\varphi_1\, \inner{\varphi_2\wedge \varphi_3\wedge \varphi_4, A} - \varphi_2\, \inner{\varphi_1\wedge \varphi_3\wedge \varphi_4, A}  \\
		& + \varphi_3\, \inner{\varphi_1\wedge \varphi_2\wedge \varphi_4, A} - \varphi_4\, \inner{\varphi_1\wedge \varphi_2\wedge \varphi_3, A}.
	\end{align*}
\end{example}

An inner/Hermitian product $\inner{\cdot,\cdot}$ in $X$ gives the musical isomorphism $\flat:X\rightarrow X'$, 
$v^\flat(w) = \inner{v,w}$ for $v,w\in X$, whose outermorphism 
	\CITE{[p.\,48]{VazJr2016}}
enables contractions $A\lcontr B = A^\flat \lcontr B$ and $B\rcontr A = B \rcontr A^\flat$ of multivectors $A,B\in\bigwedge X$.
	\SELF{the inverse musical isomorphism $\sharp:\bigwedge X'\rightarrow \bigwedge X$ gives contractions between forms, $\alpha \lcontr \beta = \alpha^\sharp \lcontr \beta$ and $\beta\rcontr\alpha = \beta \rcontr \alpha^\sharp$ for $\alpha,\beta\in\bigwedge X'$}
Though not so common outside of GA,
	\SELF{Possibly due to the need for an inner product. Or a symplectic form, but this does not seem to have been studied. Shaw1983 defines Hodge star in symplectic space, but not contraction}
they are simpler and have more direct geometric interpretations,
just as it is easier to work with inner product spaces than dual ones.
In the complex case, $\flat$ is conjugate-linear, so these contractions are sesquilinear, while those of multivectors with forms were bilinear.
This construction is equivalent to \Cref{df:contraction}:
we have $\inner{A^\flat,B} = \inner{A,B}$, where the first $\inner{\cdot,\cdot}$ is the pairing 
and the other is the inner product,
and so
$\inner{C, A \lcontr B} = \inner{C^\flat, A^\flat \lcontr B} = \inner{A^\flat \wedge C^\flat, B} = \inner{( A\wedge C)^\flat, B} = \inner{ A\wedge C, B}$ for $C\in\bigwedge X$.

\subsection{Different conventions}\label{sc:Other Conventions}

If having left and right contractions between different kinds of elements is not confusing enough, one must be aware of the various conventions.
To make matters worse, usually these are not clearly identified. 
For simplicity, here $A,B,C$ can be multivectors or forms.

Most authors use $\lcontr$ for the left contraction, with the lower side of the `hook' towards the contractor.
But in \cite{Reichardt1957} it is towards the contractee; 
in \cite{Marcus1975} the contractor is on the left of either $\lcontr$ or $\rcontr$;
\cite{Federer1969,Sternberg1964} use $\lcontr$ to contract a multivector on a form, and $\rcontr$ for the opposite (but their $\lcontr$'s differ).
	\CITE{In Sternberg1964 both are adjoints of left $\wedge$, and called right interior multiplications.
	In Federer1969 $\lcontr$ is adjoint of right $\wedge$}
In Differential Geometry \cite{Abraham1988,Kobayashi1996a}, 
contraction by a vector $v$ often appears as an operator $\mathrm{i}_v$ or $\iota_v$.
Other symbols used are $\llcontr$ \cite{Shaw1983}, $\glcontr$ \cite{Dorst2007}, $\dashv$ \cite{BayroCorrochano2018}, $\ominus$ \cite{Browne2012} and $:$ \cite{Pavan2017}.
We use $\lcontr$, $\llcontr$, $\glcontr$ to distinguish conventions I, II, III below, but this is not common practice.
	\CITE{Shaw1983 uses llcontr for I} 
In GA, some authors \cite{Dorst2007,Hitzer2012,VazJr2016} use $\glcontr$, but there is an effort to standardize $\lcontr$ (in our opinion, this is unfortunate, as $\glcontr$ helps identify their convention).

Definitions and properties differ as well.
\Cref{tab:conventions} shows how some properties of the left contraction vary in common conventions.
We use convention I \cite{Greub1978,Rosen2019,Shaw1983}.
In II, used by many authors \cite{Federer1969,Fulton1991,Gallier2020,Kozlov2000I,Lima2009},
sides are switched: their left contraction $A\llcontr B$ is our right one $B \rcontr A$, and vice-versa.
	\CITE{Sternberg1964 uses I for $\lcontr$, II for $\rcontr$}
In III, used in GA \cite{Dorst2007,Hitzer2012,Lounesto1993} and by Bourbaki \cite{Bourbaki1998}, there is a reversion $\tilde{\ }$ in the contractor: its $A\glcontr B$ is our $\tilde{A}\lcontr B$.

\begin{table}[]
	\centering
	\begin{tabular}{@{}lll@{}}
		\toprule
		I   & \makecell[l]{ (\emph{a}) \\ (\emph{b}) \\ (\emph{c})} & \makecell[l]{$\inner{A\lcontr B, C} = \inner{B, A\wedge C}$ \\ $(A\wedge B)\lcontr C = B\lcontr (A\lcontr C)$ \\ $v\lcontr (B\wedge C) = (v\lcontr B)\wedge C + (-1)^{\grade{B}} B\wedge (v\lcontr C)$}  \\ \midrule
		II  & \makecell[l]{ (\emph{a}) \\ (\emph{b}) \\ (\emph{c})}  & \makecell[l]{$\inner{A\llcontr B, C} = \inner{B, C\wedge A}$ \\ $(A\wedge B)\llcontr C = A\llcontr (B\llcontr C)$ \\ $v\llcontr (B\wedge C) = (-1)^{\grade{C}} (v\llcontr B)\wedge C +  B\wedge (v\llcontr C)$}  \\ \midrule
		III & \makecell[l]{ (\emph{a}) \\ (\emph{b}) \\ (\emph{c})}  & \makecell[l]{$\inner{A\glcontr B, C} = \inner{B, \tilde{A}\wedge C}$ \\ $(A\wedge B)\glcontr C = A\glcontr (B\glcontr C)$ \\ $v\glcontr (B\wedge C) = (v\glcontr B)\wedge C + (-1)^{\grade{B}} B\wedge (v\glcontr C)$}  \\ 
		\bottomrule
	\end{tabular}
	\caption{Properties of the left contraction in conventions I, II and III.}
\label{tab:conventions}
\end{table}

In I (\resp II), property (\emph{a}) means the left contraction by $A$ is the adjoint of the left (\resp right) exterior product by $A$.
In III, it is the adjoint of the left exterior product by the reversed $\tilde{A}$.

In II and III, (\emph{b}) shows $\bigwedge X$ is a left $(\bigwedge X)$-module \wrt the left contraction.
In I, it is a right module, as the order of $A$ and $B$ is reversed, but the notation does not make this evident.

In I and III, (\emph{c}) is a graded Leibniz rule, but in II the sign is at the `wrong' term. 
In I and III it may seem misplaced for the right contraction, 
$(B\wedge C)\rcontr v = (-1)^{\grade{C}} (B\rcontr v)\wedge C +  B\wedge (C\rcontr v)$,
but becomes natural if we think of $v$ as `coming from the right', as the notation suggests. 

Contractions I, II and III differ only by grade dependent signs, so which one to use is a matter of choice.
But fixing a standard one, preferably that with more intuitive formulas, would reduce the confusion.

We advocate for I. Some authors see the reordering of $A$ and $B$ in (\emph{b}) as a drawback, hiding the module structure. 
But thinking in terms of modules seems to bring little advantage here, while the reordering fits well with the nature of the contraction as an adjoint operator.

The weird `Leibniz rule' of II is for us a deal breaker. Many authors seem content with it, but its popularity may be an accident of history: according to \cite[p.\,112]{Gallier2020}, Bourbaki used II in the 1958 edition of \cite{Bourbaki1998}, which might explain its early dissemination.

Bourbaki's switch to III in the 1970 edition seems to have been ill-assimilated, and III only became popular with its use in GA.
The reversion%
\footnote{Bourbaki used instead a homomorphism into the opposite algebra.}
in (\emph{a}) enforces the left module structure while preserving the Leibniz rule,
but it makes orientations harder to interpret, as we discuss below.

\subsection{Geometric algebra contractions}\label{sc:geometric algebra}

The contractions of III were introduced in GA by Lounesto \cite{Lounesto1993},
	\CITE{p.221}
with the reversion $\tilde{\ }$ used ``to absorb some inconvenient signs'' \cite[p.\,134]{Dorst2001}.
But these force their way back, requiring more adjustments: 
e.g., our formula $A \lcontr B = \inner{A,B}$, for equal grades, becomes $A\glcontr B = A*B$, with $\tilde{\ }$ hidden in a \emph{scalar product} $A*B = \inner{\tilde{A},B}$;
then (\emph{a}) becomes $C*(A\glcontr B) = (C\wedge A)*B$ \cite[p.\,38]{Dorst2002}, with $A$ at the right side of $\wedge$; and so on.

Convention III serves another purpose in GA:
for $A\in\bigwedge^p X$ and $B\in\bigwedge^q X$, it lets $A\glcontr B = \comp{AB}{q-p}$ and $A\grcontr B = \comp{AB}{p-q}$ be components of the Clifford product $AB$, as are other GA products: $A*B = \comp{AB}{0}$ 
and $A\wedge B = \comp{AB}{p+q}$.
But $AB$ reflects the orientations of $\tilde{A}$ and $B$ \cite{Mandolesi_Products}, so to have results with orientations directly related to those of $A$ and $B$ one must often use $\tilde{A}B$:
e.g., $\|A\|^2 = \tilde{A}*A = \comp{\tilde{A}A}{0}$.
The exterior product is not affected by this  \cite[p.\,25]{Mandolesi_Products}, but contractions are.

Interpreting the orientation of $A\glcontr B =  (-1)^{\frac{p(p-1)}{2}} A\lcontr B$ is less immediate than in \Cref{pr:characterization 1}, taking some thought and knowledge of $p$.
For example, $(i\wedge j)\glcontr(i\wedge j\wedge k) = -k$ for the canonical basis of $\R^3$, an algebraically easy result, but with a sign whose meaning is not obvious.
The only interpretation for the orientation of $A\glcontr B$ we could find (\cite[p.\,76]{Dorst2007}, \cite[p.\,29]{dorst2002geometric}, \cite[p.\,45]{Lounesto2001})
	\CITE{figuras}
is for $p=1$, when $A\glcontr B = A \lcontr B$.

Another case of signs gone awry in GA is the \emph{dual} $A^*=A\glcontr \tilde{\Omega}$ (for a unit pseudoscalar $\Omega$ in Euclidean $\R^n$), 
which differs by $(-1)^\frac{p(p-1)+n(n-1)}{2}$ 
	\SELF{Corresponding to a reversion of both the multivector and the volume element}
from the usual Hodge dual: 
	\SELF{and our $\star_R$}
e.g., $i^*=-j\wedge k$ for $\Omega = i\wedge j\wedge k$ (strangely, a figure in \cite[p.\,82]{Dorst2007} presents this as the usual right-hand rule).

For beginners, these signs with no obvious meaning 
are an off-putting aspect of GA.
One soon learns to put some to good use: \eg identifying $\C$ with $\bigwedge^+ \R^2$, with imaginary unit $I=i\wedge j$, as $I^2 = -1$.
But most signs remain a nuisance:
$I*I = -1$ gives no new information, is less useful since $*$ is not as flexible as the Clifford product,
and to interpret its sign one must stop and think that $I*I = (-1)^{\frac{p(p-1)}{2}} \inner{I,I}$.
Such details, and the use, with altered meanings, of misleadingly familiar symbols and terms, might explain why GA is still not as widely used as it should be.

It seems the theory was thrown a little off-track by the idea of all products being components of $AB$, perhaps due to a sense of algebraic elegance trumping geometric interpretation. 
Its intuitiveness and familiarity might improve if instead of $A*B$ and $A\glcontr B$ we use $\inner{A,B}$ and $A\lcontr B$, 
which are components of $\tilde{A}B$.
This requires adapting some formulas (e.g., the projection $P_B  A = (A\glcontr B)\glcontr B^{-1}$ becomes $P_B A =\frac{B\rcontr A\lcontr B}{\|B\|^2}$),
but does not seem to cause a loss of computational power.
Lounesto \cite{Lounesto2001} uses $\inner{A,B}$, and
Dorst \cite[p.\,71]{Dorst2007} has suggested absorbing  $\tilde{\ }$ into $A*B$,
	\CITE{Chisolm2012 p.35 does define $A*B = \comp{\tilde{A} B}{0}$, so it equals $\inner{A,B}$}
but as they still use $A\glcontr B = \tilde{A}\lcontr B$, this half-way solution becomes less convenient.
Rosén \cite{Rosen2019} uses $\inner{A,B}$ and $A\lcontr B$, but with multi-covectors.

Another product that should be avoided is \emph{Hestenes inner product} \cite{Hestenes1984}, a symmetrized contraction\footnote{In \cite{Hestenes1984}, Hestenes defined $A\cdot B = 0$ if $A$ or $B$ were scalars, but later changed it \cite{Hestenes2005}.}: 
$A\cdot B = \comp{AB}{|q-p|} = A\glcontr B$ if $p\leq q$, $A\grcontr B$ otherwise.
The symmetry $A\cdot B = \pm B\cdot A$, for homogeneous elements, seems handy, but blurs the distinction between contractor and contractee.
One must compare grades to know the role of each term,
which affects how contractions operate with exterior products.
So, formulas with $\cdot$ often carry grade conditionals:
e.g., equations (1.25b) and (1.25c) in \cite[p.\,7]{Hestenes1984} 
	\CITE{manter pág., numeração de eqs. repete a cada capítulo}
give different results for $A \cdot (B \cdot C)$, depending on the grades (compare with our Propositions \ref{pr:main tools}\ref{it:wedge contractor} and \ref{pr:assoc lcontr rcontr}).
Also, the adjoint duality of formulas with $\lcontr$ and $\wedge$ is partially lost with $\cdot$ (e.g., compare (1.42) and (1.43) in \cite[p.\,12]{Hestenes1984} with Propositions \ref{pr:main tools}\ref{it:Leibniz vector grade inv} and \ref{pr:v wedge (M contr N)}, which have no grade restrictions).
Grade conditionals hamper the use of non-homogeneous elements (which are an intrinsic part of GA, arising from Clifford products), 
and force us to track grades
and analyze various grade dependent cases in proofs.
The asymmetry of contractions, which appropriately vanish when grade conditions are not satisfied, often lets us treat all cases at once, allowing simpler proofs for more general results (as can be seen throughout this work).
For more discussions of the advantages of contractions over Hestenes product, see \cite[pp.\,134--136]{Dorst2001}, \cite[pp.\,39,45]{Dorst2002}, \cite[pp.\,224--225]{Lounesto1993} or \cite[pp.\,22--24]{Mandolesi_Products}.


\providecommand{\bysame}{\leavevmode\hbox to3em{\hrulefill}\thinspace}
\providecommand{\MR}{\relax\ifhmode\unskip\space\fi MR }
\providecommand{\MRhref}[2]{%
	\href{http://www.ams.org/mathscinet-getitem?mr=#1}{#2}
}
\providecommand{\href}[2]{#2}

\end{document}